\documentclass[11pt,dvipsnames]{amsart}
\usepackage{xy}
\usepackage{graphicx, amsmath, amssymb, amsthm,amsfonts}
\usepackage[mathscr]{eucal}
\usepackage[pagebackref, colorlinks, citecolor=Red, linkcolor=NavyBlue, urlcolor=NavyBlue]{hyperref}
\hypersetup{linktoc=all}
\usepackage{comment}
\usepackage{hyperref}
\usepackage[capitalise, nameinlink]{cleveref}
\usepackage{microtype}
\usepackage{tikz-cd}
\usepackage{spectralsequences}

\newcommand{\C}{\mathbb{C}}
\newcommand{\Q}{\mathbb{Q}}
\newcommand{\Z}{\mathbb{Z}}

\newcommand{\F}{\mathbb{F}}

\newtheorem{theorem}{Theorem}[section]
\newtheorem{lemma}[theorem]{Lemma}
\newtheorem{proposition}[theorem]{Proposition}
\newtheorem{corollary}[theorem]{Corollary}

\newtheorem{theoremLET}{Theorem}

\newtheorem{propositionLET}[theoremLET]{Proposition}
\newtheorem{corollaryLET}[theoremLET]{Corollary}

\newtheoremstyle{BoldRemark} 
{10pt}                       
{10pt}                       
{\upshape}                   
{}                           
{\bfseries}                  
{.}                          
{.5em}                       
{}  
\theoremstyle{BoldRemark}
\newtheorem{remark}[theorem]{Remark}
\newtheorem{example}[theorem]{Example}
\newtheorem{definitionLET}[theoremLET]{Definition}
\newtheorem{definition}[theorem]{Definition}
\newtheorem{notation}[theorem]{Notation}
\newtheorem{conjecture}[theorem]{Conjecture}
\newtheorem{problem}[theorem]{Problem}

\usepackage[textwidth=25mm, textsize=tiny]{todonotes}

\newcommand{\Res}{\mathrm{Res}}

\newcommand{\Coind}{\mathrm{CoInd}}

\newcommand{\FP}{\mathrm{FP}}

\newcommand{\Mod}{\mathrm{Mod}}
\newcommand{\Mack}{\mathrm{Mack}}
\newcommand{\Tamb}{\mathrm{Tamb}}
\newcommand{\Green}{\mathrm{Green}}
\newcommand{\Set}{\mathrm{Set}}
\newcommand{\Ring}{\mathrm{Ring}}

\usepackage[margin=1in]{geometry}

\title{Clarification and coinduction of Tambara functors}
\author{Noah Wisdom}
\date{}

\begin{document}
\begin{abstract}
    Tambara functors are equivariant analogues of rings arising in representation theory and equivariant homotopy theory. We introduce the notion of a clarified Tambara functor and show that under mild conditions every Tambara functor admits a decomposition as a product of coinductions of clarified Tambara functors; projection onto the non-coinduced part defines a reflective localization we call clarification. Through this perspective we study Morita invariance and $K$-theory of Tambara functors, field-like Tambara functors, and Nullstellensatzian clarified Tambara functors. 
\end{abstract}

\maketitle
\tableofcontents

\section{Introduction}

Equivariant algebra concerns the study of certain generalizations of finite group actions on standard algebraic objects: Mackey functors generalize abelian groups and modules with $G$-action, and Green and Tambara functors generalize rings with $G$-action, for $G$ a finite group. Originally introduced by Dress \cite{Dre71}, Green \cite{Gre71}, and Tambara \cite{Tam93} to organize structure present in group cohomology and representation rings among many other things, the basic objects of equivariant algebra have turned up surprisingly often in equivariant homotopy theory. In particular, Brun has observed \cite{Bru05} that the homotopy groups of a $G$-$\mathbb{E}_\infty$ ring spectrum inherit the structure of a Tambara functor, and the norm structures of Tambara functors are known to be closely related to the Hill-Hopkins-Ravenel norms involved in the solution of the Kervaire invariant one problem \cite{HHR}. 

We focus our attention on $G$-Tambara functors. Roughly speaking, these may be thought of as generalizations of Galois field extensions $K \rightarrow L$ with Galois group $G$. In this situation, for each subgroup $H$ of $G$ we have a fixed-point intermediate field $L^H$, inclusions and conjugations between the various fixed-point fields, residual Weyl group $W_G H$ actions on fixed-points $L^H$, and Galois-theoretic norm and trace maps between the fixed-point subfields. A Tambara functor may be thought of as relaxing the condition that the inclusions of subfields are injective, and additionally relaxing the assumption that everything in sight is a field. In particular, a Tambara functor $R$ consists of the data of rings $R(G/H)$ for each subgroup $H$ of $G$, with contravariant functoriality along maps of transitive $G$-sets, and covariantly functorial norms and transfers satisfying many conditions. A Green functor is just a Tambara functor without any norms.

The results of the present article arose out of attempts to generalize some of the results of the author's previous work \cite{Wis24}. Namely, while Tambara functors are equivariant algebra analogues of rings, the notion of a field-like Tambara functor \cite{Nak11a} is an equivariant analogue of a field. There is a certain functor, coinduction, denoted by $\mathrm{Coind}_H^G$, from the category of $H$-Tambara functors to the category of $G$-Tambara functors, and in \cite{Wis24}, the author shows that when $G = C_{p^n}$ is cyclic of prime power order, every field-like $G$-Tambara functor is the coinduction of a particularly nice kind of field-like $G$-Tambara functor. Coinduction admits a concrete description, is a right adjoint, and enjoys many other nice properties, so this gives a good handle on field-like $C_{p^n}$-Tambara functors.

There is also a coinduction functor from $H$-rings to $G$-rings, which sends an $H$-ring $R$ to the product $R^{\times |G/H|}$ with a $G$-action that transtively permutes the factors. This functor is very closely related to coinduction for Tambara functors, and one of our main results shows that this relationship is even closer than it seems at first glance. In short, one may detect that a Tambara functor is coinduced merely by checking that its $G/e$ level is coinduced as a $G$-ring.

\begin{theoremLET}[cf. \cref{prop:coinduction-is-detected-at-bottom}]
	Let $R$ be a $G$-Tambara functor. If the $G$-ring $R(G/e)$ is isomorphic to $\mathrm{Coind}_H^G S$ for some $H$-ring $S$, then there is a $G$-Tambara functor isomorphism $R \cong \mathrm{Coind}_H^G T$ for some $H$-Tambara functor $T$ with $T(H/e) \cong S$.
\end{theoremLET}

From this we immediately obtain a generalization of one of the main results of \cite{Wis24}, from $G = C_{p^n}$ a cyclic group of prime power order to arbitrary finite $G$.

\begin{corollaryLET}
	Let $k$ be a field-like $G$-Tambara functor. Then $k \cong \mathrm{Coind}_H^G \ell$ for some field-like $H$-Tambara functor $\ell$ such that $\ell(H/e)$ is a field.
\end{corollaryLET}

One might wonder what happens when $R(G/e)$ is not coinduced. If $R(G/e)$ is a Noetherian ring, then one can show that $R(G/e)$ is $G$-equivariantly isomorphic to a product over subgroups $H$ of $G$ of coinductions of $H$-rings which themselves are not coinduced. We introduce the notion of \emph{clarified} to formalize the notion of not being coinduced.

\begin{definitionLET}
	A $G$-ring $R$ is \emph{clarified} if it does not contain any idempotents $d$ with both of the following properties: 
	\begin{enumerate}
		\item the isotropy group of $d$ is a proper subgroup of $G$, and
		\item the $G$-orbits of $d$ which are not equal to $d$ are orthogonal to $d$, that is, $gd \neq d$ implies $d \cdot gd = 0$.
	\end{enumerate}
\end{definitionLET}

We call a Tambara functor $R$ clarified if the $G$-ring $R(G/e)$ is. Additionally, one may show (with a fair bit of work) that if $R(G/e)$ admits a $G$-equivariant decomposition as a product of two $G$-rings, then in fact $R$ admits a $G$-Tambara functor decomposition as a product of two $G$-Tambara functors. Combining these results, we obtain our main theorem.

\begin{theoremLET}[cf. \cref{thm:product-decomposition}]
	Let $R$ be a Tambara functor such that $R(G/e)$ is Noetherian. Then there exists an isomorphism 
    \[ 
        R \cong \prod_{H \subset G} \mathrm{Coind}_H^G R_H 
    \] 
    such that $R_H$ is a clarified $H$-Tambara functor.
\end{theoremLET}

The norms present in Tambara functors are actually essential in the proof: we observe that \cref{thm:product-decomposition} fails for Green functors, which are essentially Tambara functors without norms, in \cref{ex:product-splitting-fails-for-Green-functors}. In fact, the failure is pretty catastrophic, as there are multiple distinct ingredients in the proof which make very heavy use of the norms present in a Tambara functor.

This result is actually more powerful than it might initially seem, as any map into a Tambara functor which may be written as a product in this way becomes very understandable. Indeed, we know how to map into categorical products, via their universal property. Additionally, coinduction is right adjoint to a very simple functor: restriction. If the domain of our morphism also admits a product decomposition of the above form, then we can push our analysis yet further. One may show that coinduced Tambara functors admit no maps into clarified Tambara functors, roughly speaking. In some sense we may therefore view \cref{thm:product-decomposition} as providing a basis in which every Tambara functor morphism is upper triangular. If our morphism is an automorphism, then we observe in \cref{prop:product-thm-diagonalizes} that it must be diagonal.

As evidence that clarified Tambara functors are a natural object to consider, we provide a lengthy list of examples of them. Additionally, we observe that Tambara functors which fit into global Tambara functors or global power functors tend to be clarified, although we decline to formalize this observation into a proposition. This is mainly due to the lack of sufficient technical foundations on these objects; in the presence of these foundations, it should be very easy to prove that any Tambara functor arising from a global object is clarified.

\begin{propositionLET}[cf. \cref{prop:examples}]
	The following Tambara functors are clarified.
	\begin{enumerate}
		\item The Burnside Tambara functor of \cref{ex:Burnside-Tamb-functor}.
		\item The representation Tambara functor of \cref{ex:representation-Tamb-functor}.
		\item Any constant Tambara functor (fixed-point Tambara functors associated to rings with trivial action).
		\item The fixed-point $G$-Tambara functor associated to any Galois field extension with Galois group $G$.
		\item Any free polynomial algebra over a clarified Tambara functor (cf. \cref{def:free-poly-algs}).
		\item The Tambara functor associated to the assignment $G/K \mapsto \mathrm{H}^{2*}(K;\mathrm{Z})$ of the group cohomology of $K$ with coefficients in $\mathrm{Z}$ (with the trivial action) \cite{Eve63} \cite{Tam93}.
		\item $\underline{\pi}_0(MU_G)$, for $MU_G$ the $G$-equivariant (homotopical) complex cobordism spectrum.
	\end{enumerate}
	Consequently if $G$ is nontrivial, there are no Tambara functor morphisms from any coinduced Tambara functor to one of the above form.
\end{propositionLET}

We note that coinduced Tambara functors arise ``in nature" fairly often, so that the above proposition applies to concrete examples of interest. For example, the main result of \cite{SSW24} roughly says that the algebraically closed fields in the category of $G$-Tambara functors are all coinduced. Additionally, in \cite[Theorem A]{Wis25c} it is shown that the adjunction units $k \rightarrow \Coind_H^G \Res_H^G k$ are all \'{e}tale (and in some cases these are essentially all of the finite \'{e}tale $k$-algebras, cf. \cite[Theorem C]{Wis25c}).

There are applications to other things besides morphisms of Tambara functors. One would expect products to behave nicely against most aspects of Tambara functors. What is somewhat surprising is that coinduction also behaves extremely nicely.

We warn the reader that induction and coinduction are the same for Mackey functors, but not for Green or Tambara functors. We will not require induction of Green or Tambara functors, so we opt to refer only to coinduction everywhere for uniformity, but whenever we mention the coinduction of a Mackey functor $M$, this is the same thing as the induction of $M$.

There is a notion of Mackey functor module over a Green functor, which we review in \cref{sec:review}. If $k$ is a Tambara functor, then we may forget the data of its norms, and obtain a Green functor. By a module over a Tambara functor, we mean a module over the underlying Green functor. 

The first suggestion that the result below should be true was explained to the author by Mike Hill, who asserted that it is an algebraic analogue of the main results of \cite{BDS15}. A proof of the special case $H = e$ of this result was obtained previously by Schuchardt, Spitz, and the author in \cite{SSW24}, and a proof of the special case of $G$ abelian was obtained previously by Chan and the author in \cite{CW25}.

\begin{theoremLET}[cf. \cref{prop:coind-induces-module-cat-iso}]
	Let $k$ be an $H$-Green functor and $H \subset G$. Coinduction induces a symmetric monoidal equivalence of categories, from the category of $k$-modules to the category of $\mathrm{Coind}_H^G k$-modules.
\end{theoremLET}

\cref{thm:product-decomposition} therefore may be interpreted as a statement about Morita equivalence: when $R(G/e)$ is Noetherian, then it is Morita equivalent to a collection of clarified Tambara functors (see \cref{thm:Morita-equiv} for a precise statement). In fact, more is true. This equivalence also determines an equivalence between the category of Tambara functor algebras.

\begin{corollaryLET}[cf. \cref{cor:coind-gives-equiv-of-Tamb-algebra-cats}]
	Coinduction induces an equivalence between the slice category of $H$-Tambara functors under $k$ and the slice category of $G$-Tambara functors under $\mathrm{Coind}_H^G k$.
\end{corollaryLET}

These are evidently fairly useful results: in \cite[Theorem C]{Wis25c}, the author uses \cref{thm:product-decomposition} to classify finite \'{e}tale extensions of the constant $G$-Tambara functor $\underline{\F}$ associated to any algebraically closed field $\F$, which leads to a complete description of affine \'{e}tale group schemes over $\underline{\F}$ in \cite[Corollary D]{Wis25c}. On the other hand, in \cite{Wis25b}, the author uses \cref{cor:coind-gives-equiv-of-Tamb-algebra-cats} and \cref{thm:Morita-equiv} to show that the Nakaoka spectrum of $\Coind_H^G R$ is naturally homeomorphic to the Nakaoka spectrum of $R$ for $R$ any $H$-Tambara functor and to construct a stratification of the Nakaoka spectrum of any Tambara functor (cf. \cite[Theorem A, Definition B]{Wis25b}).

Finally, we conclude with a study of Nullstellensatzian objects in the category of clarified Tambara functors. Introduced in \cite{BSY22}, Nullstellensatzian objects are generalizations of algebraically closed fields, by abstraction of Hilbert's Nullstellensatz. Previous work of the author joint with Schuchardt and Spitz classifies the Nullstellensatzian $G$-Tambara functors: they are precisely the coinductions $\mathrm{Coind}_e^G \F$ of algebraically closed fields. In fact, we are able to give a short proof of this result with the machinery we build up, which uses the same ideas as one of the two proofs given in \cite{SSW24}.

Since anything coinduced is not clarified, it is reasonable to expect Nullstellensatzian clarified Tambara functors to more closely resemble algebraically closed fields. We show that this is generically false.

\begin{theoremLET}[cf. \cref{thm:NSS-clarified-field-like-are-boring}]
	Let $k$ be a Nullstellensatzian clarified $G$-Tambara functor which is field-like and such that the characteristic of the field $k(G/G)$ does not divide $|G|$. Then $k(G/e)$ is algebraically closed.
\end{theoremLET}

Indeed, Artin-Schreier theory \cite{AS27a} \cite{AS27b} implies that either $G$ acts trivially on any algebraically closed field $k(G/e)$, or, if $G$ acts nontrivially, then the $G$-action factors through a faithful $C_2$-action and the characteristic of $k(G/e)$ is zero. For many groups $G$, it is not hard to construct Nullstellensatzian clarified $G$-Tambara functors $k$ with $G$ acting faithfully on $k(G/e)$, and we consequently obtain the following.

\begin{corollaryLET}[cf. \cref{cor:existence-of-non-field-like-NSS-clarifieds}]
	Let $G$ be a nontrivial finite group and assume there exists a finite field extension $L$ over $K$ with Galois group $G$. If $\mathrm{char}(L) \neq 0$ or $G \neq C_2$ then there exists a Nullstellensatzian clarified $G$-Tambara functor which is not field-like and whose $G/e$ level is not a field.
\end{corollaryLET}

In light of this, there are two natural questions. First, \cref{conj:FP-C-is-NSS-clarified} asserts that there are field-like Nullstellensatzian clarified $G$-Tambara functors of every form not eliminated by \cref{thm:NSS-clarified-field-like-are-boring} and Artin-Schreier theory. Second, \cref{problem:what-do-they-look-like?} asks what the remaining Nullstellensatzian clarified $G$-Tambara functors look like. We are able to make slight progress along this second front.

\vspace{3mm}

\noindent \textbf{Contents.} The contents of this article are organized as follows. We begin in Section 2 with a review of the necessary background on Mackey, Green, and Tambara functors for the main results of the next two sections. In Section 3 we extract many useful consequences of the Tambara functor axioms, and prove our main theorem.

Section 4 concerns an analysis of morphisms between Tambara functors admitting a product decomposition as in \cref{thm:product-decomposition}. Additionally, we elaborate on some of the unusual behavior of coproducts in the category of clarified Tambara functors. Finally, Section 5 contains all of our applications. Beginning with an overview of applications appearing in other article, Section 5.1 studies modules and algebras over coinductions, and Section 5.2 studies Nullstellensatzian clarified Tambara functors.

\vspace{3mm}

\noindent \textbf{Acknowledgments.} The author thanks Mike Hill for a wealth of insightful suggestions, including a major rearranging of the content of the early versions of this article, and David Chan for many helpful conversations. Additionally, the author thanks David Mehrle, Natalie Stewart, Maxime Ramzi, and the wonderful participants of Northwestern's informal Number Theory seminar organized by Haochen Cheng for helpful suggestions. For providing comments and suggestions on early drafts, the author thanks Mike Hill, David Mehrle, Jackson Morris, Birgit Richter, and Yuri Sulyma. Finally, the author thanks Alley Koenig for suggesting the term ``clarified" as a replacement for ``pure".
\section{Review of equivariant algebra}\label{sec:review}

There are many wonderful introductions to the theory of Mackey functors, for example \cite{Web00} and \cite{Str12}. Originally formulated by Dress \cite{Dre71} and Green \cite{Gre71} in the context of algebra, their appearance in equivariant homotopy theory has motivated their study by topologists as well. We begin with a brief review of Mackey functors here. Let $G$ be a finite group.

\begin{definition}
	A Mackey functor is a product-preserving functor $M : \mathrm{Span}(G-\Set^{\mathrm{fin}}) \rightarrow \Set$ from the Burnside category to sets, for which each monoid $M(X)$ is actually a group. A morphism of Mackey functors is a natural transformation.
\end{definition}

Given a finite $G$-set $X$ and semi-Mackey functor $M$, we refer to $M(X)$ as the value of $M$ in level $X$. Given a map of $G$-sets $f : X \rightarrow Y$, there are two natural bispans one may form: $X \leftarrow X \rightarrow Y$ and $Y \leftarrow X \rightarrow X$. The image of the former morphism under $M$ is called \emph{transfer} and the image of the latter under $M$ is called \emph{restriction}.

\begin{example}
	The Burnside Mackey functor $\mathcal{A}_G$ is the group completion of the functor $X \mapsto \{ Y \rightarrow X \}$. Concretely, $\mathcal{A}_G(G/H)$ is the Grothendieck group completion of the set of finite $H$-sets under disjoint union.
\end{example}

\begin{example}
	Let $E$ be a genuine $G$-spectrum. Then, for each $n$, $G/H \mapsto \pi_n^G(G/H_+ \wedge E)$ determines a Mackey functor $\underline{\pi}_n(E)$. According to Peter May, this fact was observed and promoted by Gaunce Lewis.
\end{example}

The \emph{box product} $\boxtimes$ of two Mackey functors is given by the Day convolution product with respect to the product of finite $G$-sets and tensor product of abelian groups. The box product makes the category of Mackey functors into a symmetric monoidal category with unit the Burnside Mackey functor. The value of the box product at the bottom level, $G/e$, is particularly pleasant. Namely, we have $(M \boxtimes N)(G/e) \cong M(G/e) \otimes N(G/e)$, with the Weyl group $G \cong W_G e$ acting diagonally.

Green functors are an intermediate step between Mackey functors and Tambara functors, which capture some of the multiplicative structure present in many of our examples.

\begin{definition}
	A Green functor $R$ is a Mackey functor which is a commutative monoid with respect to the box product. A morphism of Green functors is a morphism of Mackey functors which respects the monoid structure. The category of $G$-Green functors is denoted by $G$-$\Green$.
\end{definition}

\begin{example}
	Since the Burnside Mackey functor is the unit with respect to the box product, it inherits a unique Green functor structure making it into the initial object in the category of Green functors. It is given in level $G/H$ by the standard ring structure on the Burnside ring (which is induced by the product of $H$-sets).
\end{example}

Additionally, one may define modules over a Green functor in terms of the box product.

\begin{definition}
	A module over a Green functor $R$ is a module over the monoid $R$ with respect to the box product. Morphisms of modules are morphisms of Mackey functors which respect the module structure.
\end{definition}

The category of $R$-modules enjoys a symmetric monoidal structure induced by the box product, with unit given by $R$.

\begin{definition}
	Let $M$ and $N$ be $R$-modules. Their box product over $R$ is defined by \[ M \boxtimes_R N := \mathrm{Coeq}( M \boxtimes R \boxtimes N \rightrightarrows M \boxtimes N ) \] where the coequalizer is taken over the respective action maps for $M$ and $N$.
\end{definition}

We recall more standard facts about Green functors and modules.

\begin{proposition}
	The category of Green functors and the category of modules over a Green functor are bicomplete. Limits and filtered colimits are computed levelwise.
\end{proposition}

\begin{proposition}
	The category of modules over a Green functor is abelian. Kernels and cokernels are computed levelwise.
\end{proposition}

Tambara functors are Green functors with \emph{norms}, multiplicative analogues of transfers, and were first introduced by Tambara in \cite{Tam93} under the name $TNR$-functors.

\begin{definition}
	A Tambara functor is a product-preserving functor from 
    \[ 
        R : \mathrm{bispan}(G-\Set^{\mathrm{fin}}) \rightarrow \Set 
    \] 
    such that the additive monoid $R(X)$ is actually a group for each finite $G$-set $X$.
\end{definition}

Recall that the transfer along the fold map $X \sqcup X \rightarrow X$ describes the addition on $R(X)$. Similarly, the norm along the fold map $X \sqcup X \rightarrow X$ describes the multiplicative structure on $R$. Before we may unpack this definition into concrete terms, we must take a short detour into the theory of finite $G$-sets.

Let $f : X \rightarrow Y$ be a map of $G$-sets. Pullback along $f$ defines a functor from the category of $G$-sets over $Y$ to the category of $G$-sets over $X$: $A \rightarrow Y$ is sent to $X \times_Y B \rightarrow X$. This functor has a right adjoint given by the \emph{dependent product}, it sends a $G$-set $A \rightarrow X$ over $X$ to the $G$-set $\Pi_f A \rightarrow Y$. 

It will be necessary later to posses a detailed understanding of the dependent product. As a set, it consists of the set of pairs $(y,\sigma)$ where $y \in Y$ and $\sigma : f^{-1}(y) \rightarrow A$ such that the composition $f^{-1}(y) \xrightarrow{\sigma} A \rightarrow X$ is the inclusion $f^{-1}(y) \subset X$. The action of $G$ is defined by $g(y,\sigma) = (gy,g\sigma)$ where $(g \sigma)(x) = g \sigma( g^{-1} x)$. The map $\Pi_f A \rightarrow Y$ is given by $(y,\sigma) \mapsto y$.

Note furthermore that there is a map $X \times_Y \Pi_f A \rightarrow A$ given by sending a pair $(x,(y,\sigma))$ to $\sigma(x)$. One may check that the diagram
\[ \begin{tikzcd}
	X \arrow[d] & A \arrow[l] & X \times_Y \Pi_f A \arrow[l] \arrow[d] \\
	Y & & \Pi_f A \arrow[ll]
\end{tikzcd} \] 
commutes. Any diagram of $G$-sets isomorphic to one of the above form is called an \emph{exponential diagram}.

\begin{proposition}
	A Tambara functor consists of a Green functor $R$ along with norm maps $\mathrm{Nm}_K^H : R(G/K) \rightarrow R(G/H)$ for any subgroup inclusion $K \subset H$ (which we extend in the unique product-preserving way to arbitrary morphisms of finite $G$-sets) which are multiplicative monoid morphisms, and which satisfy the following relations.
	\begin{enumerate}
		\item Norms intertwine conjugation: \[ \mathrm{Nm}_{gKg^{-1}}^H c_g = c_g \mathrm{Nm}_K^H \]
		\item Norms are functorial: \[ \mathrm{Nm}_K^K = \mathrm{Id}_K \quad \textrm{and} \quad \mathrm{Nm}_K^H \mathrm{Nm}_L^K = \mathrm{Nm}_L^H \]	
		\item Restrictions and norms satisfy the double coset formula:
		\[ \mathrm{Res}_L^H \mathrm{Nm}_K^H = \prod_{g \in L \backslash H / K} \mathrm{Nm}_{L \cap gKg^{-1}}^L c_g \mathrm{Res}_{gLg^{-1} \cap K}^K \]
		\item The exponential formula holds: if 
        \[ \begin{tikzcd}
            X \arrow[d] & A \arrow[l] & X \times_Y \Pi_f A \arrow[l] \arrow[d] \\ 
            Y & & \Pi_f A \arrow[ll] 
        \end{tikzcd} \] 
        is an exponential diagram, then we assume that
        \[ \begin{tikzcd}
            R(X) \arrow[d, "\mathrm{Nm}"] & R(A) \arrow[l, "\mathrm{Tr}"] \arrow[r, "\mathrm{Res}"] & R(X \times_Y \Pi_f A) \arrow[d, "\mathrm{Nm}"] \\
            R(Y) & & R(\Pi_f A) \arrow[ll, "\mathrm{Tr}"]
        \end{tikzcd} \] 
        commutes.
	\end{enumerate}
\end{proposition}

\begin{example}
	If $G = C_p$ is the cyclic group of order $p$, we may describe a $C_p$-Tambara functor $R$ by a \emph{Lewis diagram}:
    \[ \begin{tikzcd}
        R(C_p/C_p) \arrow[d, "\mathrm{Res}"] \\
        R(C_p/e) \arrow[u, bend left = 45, "\mathrm{Tr}"] \arrow[u, bend right = 55, swap, "\mathrm{Nm}"] \arrow[loop below, "C_p"]
    \end{tikzcd} \]
Omitting some of the arrows allows us to give a similar presentation of a $C_p$-Green or $C_p$-Mackey functor.
\end{example}

In a Lewis diagram, the bottom of the diagram is the $G/e$ level, whereas the top of the diagram is the $G/G$ level. In general one thinks of the levels $G/H$ with $H$ ``small" as being the bottom levels of a Mackey functor, and the levels $G/H$ with $H$ ``large" as being the top levels.

\begin{example}\label{ex:Burnside-Tamb-functor}
	Coinduction of an $H$-set to a $G$-set defines norm maps which make the Burnside Green functor into a Tambara functor. For $G = C_p$ the Burnside Tambara functor has the following Lewis diagram:
    \[ \begin{tikzcd}
        \Z[x]/(x^2-px) \arrow[d, "x \mapsto p"] \\
        \Z \arrow[u, bend left = 70, "1 \mapsto x"] \arrow[u, bend right = 70, swap, "a \mapsto a+\frac{a^p-a}{p}x"] \arrow[loop below, "C_p-\mathrm{trivial}"]
    \end{tikzcd} \]
\end{example}

\begin{example}\label{ex:G-E_oo-gives-Tamb-functor}
	Let $E$ be a $G$-$\mathbb{E}_\infty$ ring spectrum. Then $\underline{\pi}_0(E)$ inherits the structure of a $G$-Tambara functor \cite{Bru07}.
\end{example}

\begin{example}\label{ex:representation-Tamb-functor}
	The complex representation Tambara functor $\mathrm{RU}_G$ inherits the structure of a Tambara functor, where the norms are given by the induction of an $H$-representation to a $G$-representation. Equivalently, the norms arise from the $G$-$\mathbb{E}_\infty$ ring structure on $KU_G$. More generally for any ring $R$ the assignment $G/H \mapsto K_0(R[H])$, the Grothendieck group completion of the category of $H$-representations over $R$, forms a Tambara functor.
\end{example}

\begin{example}
	Let $G$ act on a ring $R$. Then the fixed-point Green functor $\FP (R)$ admits a unique Tambara functor structure, with norms determined by the double coset formula. This construction is right adjoint to the functor sending a Green or Tambara functor $R$ to the $G$-ring $R(G/e)$. In fact, the same is true for Mackey functors: $M \mapsto M(G/e)$ is left adjoint to the fixed-point Mackey functor construction.
\end{example}

\begin{definition}[\cite{BH18}]\label{def:free-poly-algs}
	The free polynomial $G$-Tambara functor on a $G$-set $X$ is the representable Tambara functor \[ \mathrm{bispan}(G-\Set^{\mathrm{fin}})(X,-) . \] If $R$ is a Tambara functor, the free polynomial $R$-algebra on $X$ is \[ R \boxtimes \mathrm{bispan}(G-\Set^{\mathrm{fin}})(,-) . \]
\end{definition}

If $X = G/H$, we write $R[x_{G/H}]$ for the free polynomial $R$-algebra on $X$. By Yoneda's lemma, a map out of the free polynomial $R$-algebra on $X$ is the same data as a choice of element in level $X$ of the codomain. In particular, if $X = \sqcup G/H_i$, then the free polynomial $R$-algebra on $X$ is $\boxtimes_i R[x_{G/H_i}]$ where the box product is over $R$. We will not require these gadgets, except as examples, until \cref{subsec:NSS}.

There are well-known change-of-group functors relating $G$-Mackey, $G$-Green, and $G$-Tambara functors for varying $G$, arising from change-of-group functors relating categories of $G$-sets for varying $G$. Namely, if $H \subset G$, then the forgetful functor $\mathrm{res}_H^G$ from $H$-sets to $G$-sets has a left adjoint $\mathrm{ind}_H^G$. It sends an $H$-set $X$ to $G \times_H X$.

\begin{definition}
	If $T$ is an $G$-Mackey, $G$-Green, or $G$-Tambara functor, then the restriction $\mathrm{Res}_H^G T$ is defined by 
    \[ 
        \mathrm{Res}_H^G T := T \circ \mathrm{ind}_H^G .
    \]
\end{definition}

The induction of the $H$-set $H/K$ is isomorphic to the $G$-set $G/K$. Therefore the restriction may be thought of as extracting the bottom levels of a Mackey, Green, or Tambara functor.

\begin{definition}
	Let $T$ be an $H$-Mackey, $H$-Green, or $H$-Tambara functor. The coinduction $\mathrm{Coind}_H^G T$ is defined by \[ \mathrm{Coind}_H^G T := T \circ \mathrm{res}_H^G . \]
\end{definition}

Being coinduced from $H \subset G$ is a very special case of the equivalent conditions of \cite[Corollary 4.5]{Lew80}. We collect some well-known properties of restriction and coinduction.

\begin{proposition}
	Coinduction (of Mackey, Green, or Tambara functors) is right adjoint to restriction.
\end{proposition}

\begin{proposition}
	Coinduction of Mackey functors is also left adjoint to restriction.
\end{proposition}

There is another functor dual to coinduction, called induction. For Mackey functors, coinduction and induction are naturally isomorphic. We will not require induction of Green or Tambara functors, and so will not consider induction further.

\begin{lemma}\label{lem:coind-is-lax-monoidal}
	The functor 
    \[ 
        \mathrm{Coind}_H^G : H {-} \Mack \rightarrow G {-} \Mack 
    \] 
    is lax symmetric monoidal and oplax symmetric monoidal.
\end{lemma}

\begin{proof}
	By, for example, \cite[Lemma 6.6]{Cha24}, $\mathrm{Res}_H^G$ is strong symmetric monoidal. Now $\mathrm{Coind}_H^G$ is both left and right adjoint to $\mathrm{Res}_H^G$, hence is both lax symmetric monoidal and oplax symmetric monoidal.
\end{proof}

These constructions have shadows in the category of $G$-rings.

\begin{definition}
	Let $R$ be a ring with $H$-action. Define the coinduction of $R$ as the $G$-ring 
    \[ 
        \mathrm{Coind}_H^G R := \mathrm{Coind}_H^G \mathrm{FP}(R)(G/e) \cong \mathrm{Fun}(G/H,R) . 
    \]
\end{definition}

Note that $\mathrm{Fun}(G/H,R)$ is a ring via pointwise multiplication, and the $G$-action is given as follows. Fix a choice of representatives $g \in G$ for each coset in $G/H$. Let $\gamma \in G$ and $g_1 \in G$ be the chosen representative for the coset $g_1 H$, and let $g_2$ be the chosen representative of $\gamma g_1 H$, so $g_2^{-1} \gamma g_1 \in H$. Then \[ (\gamma \cdot f)(g_1 H) = g_2^{-1} \gamma g_1 \cdot (f(\gamma g_1 H)) \] defines the $G$-action on $\mathrm{Coind}_H^G R$.

\begin{proposition}
	Coinduction of $H$-rings is right adjoint to restriction $\mathrm{Res}_H^G$ from $G$-rings to $H$-rings.
\end{proposition}

We emphasize that coinduction and restriction of rings with group action form the underlying bottom level description of coinduction and restriction respectively of Tambara functors.

\begin{notation}
	If $H \subset G$ and $R$ is an $H$-Ring, Mackey, Green, or Tambara functor, then let ${}^g H := gHg^{-1}$ and ${}^g R$ be the ${}^g H$-Ring, Mackey, Green, or Tambara functor given by restricting along the conjugation isomorphism ${}^g H \cong H$.
\end{notation}

\begin{lemma}\label{lem:arbitrary-res-of-coind-of-Gring}
	Let $H$ and $K$ be subgroups of $G$ and let $R$ be an $H$-ring. Then we have an isomorphism 
\[ 
	\mathrm{Res}_K^G \mathrm{Coind}_H^G R \cong \prod_{g \in K \backslash G / H} \mathrm{Coind}_{K \cap {}^g H}^K \mathrm{Res}_{K \cap {}^g H}^{{}^g H} {}^g R 
\] 
of $K$-rings.
\end{lemma}

This will follow from the Tambara functor statement by taking fixed-point Tambara functors. We will often use \cref{lem:arbitrary-res-of-coind-of-Gring} in the special case $K = H$. The following may be found as, for example, Proposition 5.3 of \cite{TW95}, where it is stated for Mackey functors; their proof clearly goes through for Green and Tambara functors as well.

\begin{lemma} \label{lem:arbitrary-res-of-coind-of-HTamb}
	Let $H$ and $K$ be subgroups of $G$ and let $R$ be an $H$-Green, Mackey, or Tambara functor. Then we have an isomorphism 
\[ 
	\mathrm{Res}_K^G \mathrm{Coind}_H^G R \cong \prod_{g \in K \backslash G / H} \mathrm{Coind}_{K \cap {}^g H}^K \mathrm{Res}_{K \cap {}^g H}^{{}^g H} {}^g R 
\] 
of $K$-Green, Mackey, or Tambara functors respectively.
\end{lemma}

We will later see that the identity double coset factor of the above splitting controls much of the behavior of Green, Mackey, and Tambara functors.
\section{A product decomposition for Tambara functors}

In this section we prove our first main theorem, which decomposes any Tambara functor into a product of coinductions. These coinductions are maximal in a sense which is captured by the notion of a \emph{clarified} Tambara functor. The terminology comes by comparison with butter: we will see later that there is a ``clarification" functor, a localization of the category of $G$-Tambara functors, which heuristically is obtained by projecting away anything coinduced from a proper subgroup. The author imagines these coinductions as analogous to the milk solids present in melted butter.

\begin{definition}
	We call an idempotent $d$ in a $G$-ring $R$ \emph{type $H$} if $d$ has isotropy group $H$ and whenever $g \notin H$ we have $d \cdot gd = 0$.
\end{definition}

The protypical example of a type $H$ idempotent is an idempotent in $\mathrm{Coind}_H^G R$ corresponding to projection onto one of the factors $R$.

\begin{definition}
	A $G$-ring $R$ is called \emph{clarified} if $R$ contains no idempotents of type $H$, for $H$ any proper subgroup of $G$. A $G$-Tambara functor $k$ is called clarified if $k(G/e)$ is a clarified $G$-ring.
\end{definition}

Most important $G$-Tambara functors are clarified (for a list of examples, see \cref{prop:examples}). On the other hand, coinductions are not clarified in general. In particular, by the main result of \cite{SSW24}, every Nullstellensatzian $G$-Tambara functor is coinduced, hence is not clarified. There are many examples of field-like $G$-Tambara functors which are not clarified by virtue of being coinduced \cite{Wis24}. The full subcategory of $G$-Tambara functors spanned by the category of clarified $G$-Tambara functors should be thought of as a localization obtained by formally setting the images of all nontrivial coinduction functors equal to zero.

The proof of the product decomposition theorem requires a lengthy careful and technical analysis of the behaviors of idempotents and their interactions with norms. We collect many of these results in \cref{subsec:consequences-of-exp-formula}. Then in \cref{subsec:product-decomp} we prove the product decomposition theorem.

\subsection{Consequences of the exponential formula}\label{subsec:consequences-of-exp-formula}

In this subsection we collect many important consequences of the exponential formula. The key idea underlining all the proofs of this subsection is due to Mazur \cite{Maz13}. In a fixed-point Tambara functor, all exponential formulas are determined by the double coset formula. For example, the norm of a sum of elements in the $G/e$ level is given by 
\[ 
    \mathrm{Nm}_e^H(a+b) = \prod_{h \in H} \left( ha+hb \right) .
\] 
Expanding this product out, we see that two monomials give, respectively, $\mathrm{Nm}_e^H(a)$ and $\mathrm{Nm}_e^H(b)$, and every other monomial is divisible by $h_1 a \cdot h_2 b$ for some $h_i \in H$.

A transfer map is called \emph{proper} if it is induced by an inclusion $K \subset H$ of subgroups for which $K \neq H$, or is a conjugation of such.

\begin{lemma}\label{lem:norm-is-additive-up-to-proper-transfers}
	Let $a_1,...,a_n$ be any collection of elements of $R(G/K)$, for $R$ a $G$-Tambara functor, and let $L$ be a subgroup of $G$ containing $K$. Then \[ \mathrm{Nm}_K^L( \sum_i a_i ) - \sum_i \mathrm{Nm}_K^L(a_i) \] is a sum of proper transfers.
\end{lemma}

\begin{proof}
	This may be seen directly from the formula of \cite[Theorem 8.2]{AKGH21}.
\end{proof}

The following lemma is a refinement of a special case of \cref{lem:norm-is-additive-up-to-proper-transfers}. Essentially, it is the observation that the sum of proper transfers appearing in \cref{lem:norm-is-additive-up-to-proper-transfers} is zero in certain special cases.

\begin{lemma}\label{lem:norm-of-idempotents-is-additive}
	Let $a$ and $b$ be mutually orthogonal idempotents in $R(G/e)$ which are fixed by $G$, and let $H$ be any subgroup of $G$. Then $\mathrm{Nm}_e^H(a+b) = \mathrm{Nm}_e^H(a)+\mathrm{Nm}_e^H(b)$.
\end{lemma}

\begin{proof}
	The exponential formula gives an alternate description of the norm of a sum. We have 
    \[ 
        \Pi_f (G/e \sqcup G/e) \cong G/H \times \mathrm{Fun}(G/H,\{ \mathrm{l} , \mathrm{r} \}) 
    \] 
    as $G$-sets.
	
	Observe the isomorphism $G/e \times_{G/H} \Pi_f (G/e \sqcup G/e) \cong G/e \times \mathrm{Fun}(G/H,\{ \mathrm{l} , \mathrm{r} \})$, which is a disjoint union of copies of $G/e$, as $G$ acts diagonally. Each orbit has a unique representative given by the tuple with $e$ in the first coordinate. The restriction is a product of projections: it is a map \[ R(G/e) \times R(G/e) \rightarrow \prod_{(e,\sigma : H \rightarrow \{ \mathrm{l} , \mathrm{r} \})} R(G/e) \] where the map $R(G/e) \times R(G/e) \rightarrow R(G/e)$ determining the $(e,\sigma)$th factor is projection onto the left factor in the domain if $\sigma(e) = \mathrm{l}$, or the right factor if $\sigma(e) = \mathrm{r}$, followed by possibly postcomposing with left multiplication by an element of $G$. Since $a$ and $b$ are fixed by $G$ by assumption, the restriction sends $(a,b)$ to the element in the product $\Pi_{(e,\sigma : H \rightarrow \{ \mathrm{l} , \mathrm{r} \})} R(G/e)$ whose value in the $(e,\sigma)$th coordinate is $a$ if $\sigma(e) = \mathrm{l}$ and $b$ if $\sigma(e) = \mathrm{r}$.
	
	The norm along the projection onto $\Pi_f (G/e \sqcup G/e)$ is the product of a collection of norm and multiplication maps. Two factors in this product description respectively send the elements $a$ and $b$ in the factors of $R(G/e)$ corresponding to the two constant functions $\sigma$ to $\mathrm{Nm}_e^H(a)$ and $\mathrm{Nm}_e^H(b)$. To conclude the proof, it suffices to show that every other coordinate of the image of the norm along $G/e \times_{G/H} \Pi_f ( G/e \sqcup G/e) \rightarrow \Pi_f (G/e \sqcup G/e)$ is zero.
	
	The quotient map along which we are taking the norm has the form \[ (g, \sigma : gH \rightarrow \{ \mathrm{l} , \mathrm{r} \}) \mapsto (gH, \sigma : gH \rightarrow \{ \mathrm{l} , \mathrm{r} \}) . \] If $\sigma : gH \rightarrow \{ \mathrm{l} , \mathrm{r} \}$ is not constant, it suffices to find two elements $(e, \sigma_i : H \rightarrow \{ \mathrm{l} , \mathrm{r} \})$ in the preimage of the $G$-orbit of $(H, \sigma : H \rightarrow \{ \mathrm{l} , \mathrm{r} \})$ which have the property $\sigma_1(e) \neq \sigma_2(e)$, because this corresponds to the value of the norm at the factor $R(G \cdot (H,\sigma))$ of $R(\Pi_f (G/e \sqcup G/e))$ being divisible by \[ \mathrm{Nm}_e^K(a) \mathrm{Nm}_e^K(b) = \mathrm{Nm}_e^K(ab) = \mathrm{Nm}_e^K(0) = 0 \] where $K$ is the stabilizer of $(H,\sigma)$. In other words, the composition 
    \begin{align*}
        R(G/e) \times R(G/e) & \rightarrow R(G/e \times_{G/H} \Pi_f (G/e \sqcup G/e)) \\
        & \rightarrow R(\Pi_f (G/e \sqcup G/e)) \rightarrow R(G \cdot (H,\sigma))
    \end{align*} 
    sends $(a,b)$ to zero.
	
	Since $\sigma$ is not constant, there exists $h \in H$ such that $\sigma(h) \neq \sigma(1)$. Now $(e, \sigma)$ and $(h^{-1},\sigma)$ are two distinct preimages of $(H,\sigma)$, and $h(h^{-1},\sigma) = (e,x \mapsto \sigma(hx))$ implies $(e,x\mapsto \sigma(hx))$ is in the preimage of the $G$-orbit of $(H, \sigma : H \rightarrow \{ \mathrm{l} , \mathrm{r} \})$. Therefore $(e,\sigma)$ and $(e,x \mapsto \sigma(hx))$ have the desired property.
\end{proof}

The following lemma is a kind of total opposite to \cref{lem:norm-of-idempotents-is-additive}. It is a refinement of the special case of \cref{lem:norm-is-additive-up-to-proper-transfers} where the norms $\mathrm{Nm}_H^L(a_i)$ are zero but the $a_i$ sum to $1$.

\begin{lemma}\label{lem:coind-at-bottom-implies-surj-transfer}
	Let $k$ be a $G$-Tambara functor such that $k(G/e) \cong \mathrm{Coind}_H^G R$ for some $H$-ring $R$. For any $L$ which is not subconjugate to $H$, there exists a $G$-set morphism $X \rightarrow G/L$ such that the transfer $k(X) \rightarrow k(G/L)$ is surjective and the stabilizer of every point of $X$ is a proper subconjugate of $L$.
\end{lemma}

\begin{proof}
    It suffices to show that $1 \in k(G/L)$ is in the image of some transfer $k(X) \rightarrow k(G/L)$ where $X$ has the desired property, because the transfer is a map of $k(G/L)$ modules, and $1 \in k(G/L)$ generates $k(G/L)$ as a $k(G/L)$ module.
		
	Let $d$ be an idempotent of $k(G/e)$ corresponding to projection onto any one of the factors of $R$, and observe that the elements $gd$ form a complete set of orthogonal idempotents of $k(G/e)$. By \cref{lem:norm-is-additive-up-to-proper-transfers} and the fact that the norm of $1$ is $1$, it suffices to show $\mathrm{Nm}_e^L(gd) = 0$ for all $g \in G$, as we may then take $X$ to be every $G$-orbit of $\Pi_f (G/e \sqcup G/e)$ except the two corresponding to constant $\sigma$ (where $\sigma$ is as in the proof of \cref{lem:norm-is-additive-up-to-proper-transfers}). 
	
	The stabilizer of $gd$ is a conjugate of $H$, and by assumption there is an element $\gamma \in L$ which is not contained in this conjugate. Then we have $\gamma gd \neq gd$, hence $(\gamma g d)(gd) = 0$. The norm $\mathrm{Nm}_e^L$ factors through $L$-orbits, because if $\rho \in L$, then right multiplication by $\rho$ defines a $G$-equivariant morphism $G/e \rightarrow G/e$ over $G/L$ (and the $G$-action on $R(G/e)$ arises from the isomorphism $G \cong \mathrm{Aut}_{\mathrm{Span}(G-\Set^{\mathrm{fin}})}(G/e)$ given by right multiplication). Thus we compute 
    \begin{align*}
        \mathrm{Nm}_e^L & (gd) = \mathrm{Nm}_e^L(gd^2) = \mathrm{Nm}_e^L(gd)^2 \\
        & = \mathrm{Nm}_e^L(\gamma gd) \mathrm{Nm}_e^L(gd) = \mathrm{Nm}_e^L((\gamma gd)(gd)) = \mathrm{Nm}_e^L(0) = 0 
    \end{align*}
    and conclude the proof by appealing to \cref{lem:norm-is-additive-up-to-proper-transfers}.
\end{proof}

\subsection{The product decomposition}\label{subsec:product-decomp}

Our goal is to show that any $G$-Tambara functor decomposes as a product of coinductions of clarified Tambara functors. The key idea is that this decomposition is determined entirely by behavior which may be detected in the $G/e$ level. We start with the following proposition, which allows us to give part of the argument in the much simpler setting of rings with $G$-action. The author expects a similar result to be true for the global power functors \cite{Sch18} and global Tambara functors \cite{Krs20}.

\begin{proposition}\label{prop:bottom-level-reflects-products}
	The functor $R \mapsto R(G/e)$ from $G {-} \Tamb$ to $G {-} \Ring$ reflects products in the sense that if $R(G/e) \cong S_1 \times S_2$ is an isomorphism of $G$-rings, then there is a canonical isomorphism $R \cong R_1 \times R_2$ of Tambara functors with $R_i(G/e) = S_i$ given on the bottom level by the specified isomorphism $R(G/e) \cong R_1(G/e) \times R_2(G/e)$.
\end{proposition}

\begin{proof}
	Let $d_i \in S_i$ denote the unit element, which we regard as an idempotent in $R(G/e)$. Being a multiplicative monoid morphism, each $\mathrm{Nm}_e^H(d_i)$ is an idempotent. Moreover, the $\mathrm{Nm}_e^H(d_i)$ are orthogonal idempotents. By assumption, $G$ fixes each $d_i$. \cref{lem:norm-of-idempotents-is-additive} then implies \[ 1 = \mathrm{Nm}_e^H(d_1+d_2) = \mathrm{Nm}_e^H(d_1)+\mathrm{Nm}_e^H(d_2) \] so that the $\mathrm{Nm}_e^H(d_i)$ form a complete set of orthogonal idempotents in $R(G/H)$.
	
	We may now define $S_i(G/H) := \mathrm{Nm}_e^H(d_i)R(G/H)$. To check that each $S_i$ is a Tambara functor, it suffices to check that the restriction, norm, and transfer of $R$ restrict to $S_i$. Restrictions and norms take $\mathrm{Nm}_e^H(d_i)$ to $\mathrm{Nm}_e^L(d_i)$ and preserve multiplication, hence descend to $S_i$. Finally, since $\mathrm{Nm}_e^H(d_i) = \mathrm{Res}_H^L \mathrm{Nm}_e^L(d_i)$, Frobenius reciprocity therefore implies \[ T_H^L(\mathrm{Nm}_e^H(d_i) x) = \mathrm{Nm}_e^L(d_i) T_H^L(x) \] so that the transfer also descends to $S_i$.	More precisely, we have observed that the $\mathrm{Nm}_e^H(d_1)$ generate a Tambara ideal of $R$, and the corresponding quotient is given levelwise by the rings $S_2(G/H)$, and vice versa.
\end{proof}

We emphasize that \cref{prop:bottom-level-reflects-products} is particularly useful because Tambara functors tend to have much simpler $G/e$-level than any other levels. For example, in the Burnside and representation Tambara functors, the $G/e$ level is just $\Z$, whereas the other levels are more complicated finitely presented $\Z$-algebras. Thus the advantage of \cref{prop:bottom-level-reflects-products} is that it allows us to check the simplest part of a Tambara functor to see whether or not it is a product. We also note that the $G/e$ level of a nonzero Tambara functor cannot be zero, so there are no degenerate cases of \cref{prop:bottom-level-reflects-products}.

\begin{proposition}\label{prop:product-splitting-of-G-rings}
	Let $R$ be a Noetherian ring with $G$-action. Then \[ R \cong \prod_{H \subset G} \mathrm{Coind}_H^G R_H \] where each $R_H$ is a clarified $H$-ring.
\end{proposition}

\begin{proof}
	We first prove that any Noetherian ring is expressed as a finite product of rings which may not be written nontrivially as products themselves.	We will say that an idempotent $d$ ``lives under" another idempotent $d'$ if $d \cdot d' = d$.
	
	If $\{ d_i \}$ and $\{ d_j \}$ are two complete sets of orthogonal idempotents, then $\{ d_i d_j | d_i d_j \neq 0 \}$ is another complete set of orthogonal idempotents, which has cardinality at least the cardinality of $\{ d_i \}$. So, if for every $n$ there exists a complete set of orthogonal idempotents of cardinality at least $n$, then we can choose a sequence of complete sets of orthogonal idempotents $\{ d_{i,n} \}_{i \in I_n}$ such that $|I_n| \rightarrow \infty$ and each $d_{i,n}$ is a sum of idempotents $d_{j,n+1}$ (and it follows that $d_{j,n+1}$ lives under $d_{i,n}$ and does not live under any other element of $\{ d_{i,n} \}$).
	
	We construct a strictly increasing chain of ideals $J_n$ by induction. $J_n$ will be the ideal generated by $|I_n|-1$ elements of $I_n$. By inductive hypothesis, $J_{n-1}$ has been constructed so that the idempotent $d_{i,n-1}$ which it does not contain has the property that the cardinality of the subset of $I_m$ consisting of idempotents living under $d_{i,n-1}$ goes to infinity as $m$ goes to infinity. We may therefore choose an idempotent $d_{j,n}$ living under $d_{i,n-1}$ with the very same property, by the pigeonhole principle, and define $J_n$ to be the ideal generated by all idempotents of $I_n$ except $d_{j,n}$.
	
	This contradicts the assumption that $R$ is Noetherian, hence there is a finite complete set $I$ of orthogonal idempotents, such that no element of $I$ is a sum of two nontrivial idempotents. If $J$ is another such set, then the set of nonzero products of elements of $I$ and $J$ is larger than $I$, unless $I = J$. Therefore $I$ determines a decomposition of $R$ as a finite product of rings $R_i$ for which $R_i$ contains no nontrivial idempotents.
	
	Note that $G$ acts on $I$. Indeed, $G$ takes $I$ to a finite maximal complete set of orthogonal idempotents, which we have just seen must be $I$ itself. Therefore we are entitled to write $I \cong \sqcup_j G/H_j$ as a $G$-set. This immediately implies \[ R \cong \prod_j \mathrm{Coind}_{H_j}^G R_j \] where $R_j$ is the ring corresponding to the idempotent identified with the identity coset $eH_j$ under the isomorphism $I \cong \sqcup_j G/H_j$. Since coinduction commutes with products, we may rearrange and combine like terms so that this product decomposition has the claimed form.
\end{proof}

The product decomposition of \cref{prop:product-splitting-of-G-rings} is not unique. For example, when $G$ is abelian, we have a $G$ action on $\mathrm{Coind}_e^G \Z \cong \Z^{|G|}$ by permuting the factors. In fact, there may be even more automorphisms: we have $\mathrm{Coind}_e^G \left( \Z \times \Z \right) \cong \mathrm{Coind}_e^G \Z \times \mathrm{Coind}_e^G \Z$, so $G \times G$ acts faithfully on $\mathrm{Coind}_e^G \left( \Z \times \Z \right)$ when $G$ is abelian.

Combining \cref{prop:product-splitting-of-G-rings} and \cref{prop:bottom-level-reflects-products}, we are therefore reduced to showing that if $k$ is a Tambara functor with $k(G/e)$ coinduced from an $H$-ring, then $k$ is coinduced as a Tambara functor.

\begin{proposition}\label{prop:coinduction-is-detected-at-bottom}
	Let $G$ be a finite group. Let $k$ be a $G$-Tambara functor such that $k(G/e)$ is the coinduction of some $H$-ring. Then $k \cong \mathrm{Coind}_H^G \ell$ for some $H$-Tambara functor $\ell$.
\end{proposition}

The proof simplifies enormously if $H$ is normal in $G$. In this case, $\mathrm{Res}_H^G k$ carries a $G/H$ action by permuting cosets in $G/H$ according to right multiplication. One may straightforwardly show that the fixed points of this action give the desired $\ell$. The proof of \cref{prop:coinduction-is-detected-at-bottom} is a careful unpacking of this idea, abstracting it enough that the requirement that $H$ is normal disappears.

\begin{proof}
	We induct on the order of $G$. If $H = G$, there is nothing to prove, as $\mathrm{Res}_G^G$ and $\mathrm{Coind}_G^G$ are the identity functors. Now assume $H$ is a proper subgroup of $G$. The first half of our argument will consist entirely of constructing a map $k \rightarrow \mathrm{Coind}_H^G \ell$ which becomes an isomorphism after applying $\mathrm{Res}_H^G$.
	
	 To start with, \cref{lem:arbitrary-res-of-coind-of-Gring} supplies an $H$-ring isomorphism 
\[ 
	\mathrm{Res}_H^G \mathrm{Coind}_H^G R \cong \prod_{g \in H \backslash G / H} \mathrm{Coind}_{H \cap {}^g H}^H \mathrm{Res}_{H \cap {}^g H}^{{}^g H} {}^g R .
\] 
The formula $\left( \mathrm{Res}_H^G k \right) \left( H/e \right) \cong \mathrm{Res}_H^G \left( k(G/e) \right)$, \cref{prop:bottom-level-reflects-products}, and our inductive hypothesis imply that we have an isomorphism of $H$-Tambara functors
\[ 
	\mathrm{Res}_H^G k \cong \prod_{g \in H \backslash G / H} \mathrm{Coind}_{H \cap {}^g H}^H \ell_g 
\] 
where $\ell_g$ is a $(H \cap {}^g H)$-Tambara functor with the following properties. First, we have
\[ 
	\ell_g(H \cap {}^g H/e) \cong \mathrm{Res}_{H \cap {}^g H}^{{}^g H} {}^g R \mathrm{,}
\] 
and second, $\ell_g(H \cap {}^g H/L)$ receives a projection from $k(G/L)$ corresponding to the norm $\mathrm{Nm}_e^L$ of the idempotent picking out the copy of $R \cong \ell_g(H \cap {}^g H/e)$ in $k(G/e)$.

Since coinduction commutes with products, as it is a right adjoint, the restriction-coinduction adjunction yields a morphism
\[ 
	k \rightarrow \prod_{g \in H \backslash G / H} \mathrm{Coind}_{H \cap gHg^{-1}}^G \ell_g 
\] 
which we post-compose with projection onto the identity double coset factor to obtain a map $k \rightarrow \mathrm{Coind}_H^G \ell_e$.
	
	We set $\ell = \ell_e$, and we now aim to prove $\mathrm{Res}_H^G k \rightarrow \mathrm{Res}_H^G \mathrm{Coind}_H^G \ell$ is an isomorphism. Notice that the induced map in level $H/e$ is precisely the $H$-ring isomorphism of \cref{lem:arbitrary-res-of-coind-of-Gring}, as both are obtained as the adjoint of the same map. From this, we observe that the composition 
\begin{align*} 
	\prod_{g \in H \backslash G / H} \mathrm{Coind}_{H \cap {}^g H}^H \ell_g & \cong \mathrm{Res}_H^G k \\
	& \rightarrow \mathrm{Res}_H^G \mathrm{Coind}_H^G \ell \\
	& \cong \prod_{g \in H \backslash G / H} \mathrm{Coind}_{H \cap {}^g H}^H \mathrm{Res}_{H \cap {}^g H}^{{}^g H} {}^g \ell
\end{align*}
is actually diagonal. In other words, it is the product of maps 
\[ 
	\mathrm{Coind}_{H \cap {}^g H}^H \ell_g \rightarrow \mathrm{Coind}_{H \cap {}^g H}^H \mathrm{Res}_{H \cap {}^g H}^{{}^g H} {}^g \ell
\] 
and we must show each of these is an isomorphism. By construction, the $HeH$ component is an isomorphism, and the bottom level is an equivariant ring isomorphism.

	It is clear from the construction that applying $\mathrm{Res}_{H \cap {}^g H}^H$ to both maps above yields a product over double cosets of maps, and the map corresponding to the identity double coset factor is $\ell_g \rightarrow \mathrm{Res}_{H \cap {}^g H}^{{}^g H} {}^g \ell$. It suffices in fact to show that this map is an isomorphism. Unwinding definitions, we observe that this map is obtained from the structure of $k$. More explicitly, conjugation in $k$ determines an $H \cap {}^g H$-Tambara functor morphism $c_g : \ell_g \rightarrow \mathrm{Res}_{H \cap {}^g H}^H \ell_e$ by the following reasoning. The Tambara structure of both the domain and the codomain are inherited from the Tambara structure of $k$ by projecting onto norms of idempotents. Conjugation preserves this structure and sends one idempotent to the other, and hence all norms of one idempotent to all norms of the other.
	
	Our map $\ell_g \rightarrow \mathrm{Res}_{H \cap {}^g H}^{{}^g H} {}^g \ell$ is thus obtained as the composition \[ \ell_g \xrightarrow{c_g} \mathrm{Res}_{H \cap {}^g H}^H \ell_e \xrightarrow{\mathrm{Id}} \mathrm{Res}_{H \cap {}^g H}^H \ell \xrightarrow{c_{g^{-1}}} \mathrm{Res}_{H \cap {}^g H}^{{}^g H} {}^g \ell \] where the middle map is the identity double coset factor of the restriction of the map $k \rightarrow \mathrm{Coind}_H^G \ell$ and the last map is induced by conjugation by $g^{-1}$ in $\mathrm{Coind}_H^G \ell$. In fact, each of these maps is invertible: the middle map is clearly invertible, and the first and third maps have inverse induced by conjugation by $g^{-1}$ and $g$ respectively in $k$ and $\mathrm{Coind}_H^G \ell$.

	At this point, we have constructed a map $k \rightarrow \mathrm{Coind}_H^G \ell$ which becomes an isomorphism after applying $\mathrm{Res}_H^G$. Since $\mathrm{Res}_H^G$ is computed by precomposing a Tambara functor with $\mathrm{Ind}_H^G$ and $\mathrm{Ind}_H^G H/L \cong G/L$, we see that $k(X) \rightarrow \mathrm{Coind}_H^G \ell(X)$ is an isomorphism whenever $X$ is a $G$-set for which all elements of $X$ have isotropy sub-conjugate to $H$.
	
	On the other hand, assume $L \subset G$ is not conjugate to a subgroup of $H$. By \cref{lem:coind-at-bottom-implies-surj-transfer} there is a $G$-set morphism $X \rightarrow G/L$ such that all elements of $X$ have isotropy sub-conjugate to $H$ and the transfer $\mathrm{Coind}_H^G \ell(X) \rightarrow \mathrm{Coind}_H^G \ell(G/L)$ is surjective. Since $k(G/L) \rightarrow \mathrm{Coind}_H^G \ell(G/L)$ factors a composition of two surjections, $k(X) \rightarrow \mathrm{Coind}_H^G \ell(X) \rightarrow \mathrm{Coind}_H^G \ell(G/L)$, it is itself surjective.
	
	Analogously, injectivity of $k(G/L) \rightarrow \mathrm{Coind}_H^G \ell(G/L)$ will follow from the existence of an injective restriction $k(G/L) \rightarrow k(X)$, where every element of $X$ has isotropy sub-conjugate to $H$. By \cref{lem:coind-at-bottom-implies-surj-transfer}, $1 \in k(G/L)$ is the transfer of some $x \in k(X)$, where $X$ has the desired property. If $y$ restricts to zero in $k(X)$, then Frobenius reciprocity implies \[ y = 1 \cdot y = \mathrm{Tr}_X^{G/L}(x) \cdot y = \mathrm{Tr}_X^{G/L}(x \cdot \mathrm{Res}_X^{G/L}(y)) = 0 \] hence the restriction $k(G/L) \rightarrow k(X)$ is injective, as desired.
\end{proof}

This implies a strengthening of one of the main theorems of \cite{Wis24}. Namely, \cref{cor:field-likes-are-coinduced-from-clarified} below was proven only in the case $G = C_{p^n}$. It reduces the study of all field-like Tambara functors to those for which every level is a field, since this is equivalent to being clarified for field-like Tambara functors.

\begin{corollary}\label{cor:field-likes-are-coinduced-from-clarified}
	Let $k$ be a $G$-Tambara field. Then $k$ is coinduced from a clarified Tambara field $\ell$, i.e. a field-like Tambara functor whose $G/e$ level is a field.
\end{corollary}

\begin{proof}
	If $k$ is any field-like Tambara functor, then $k(G/e) \cong \mathrm{Coind}_H^G \F$ for some field $\F$ with an action of $H$ (as observed by David Chan and Ben Spitz, this is a straightforward consequence of Nakaoka's criteria \cite{Nak11a} and the Chinese remainder theorem). Therefore $k \cong \mathrm{Coind}_H^G \ell$ for some $H$-Tambara functor $\ell$ with $\ell(H/e) \cong \F$. Since the restrictions in $k$ are all injective, the restrictions in $\ell$ are all injective. Thus by Nakaoka's criteria, $\ell$ is field-like. It is visibly clarified.
\end{proof}

We may now prove the main theorem of this section.

\begin{theorem}\label{thm:product-decomposition}
	Let $R$ be a $G$-Tambara functor such that either the ring $R(G/e)$ is Noetherian. Then 
    \[ 
        R \cong \prod_H \mathrm{Coind}_H^G R_H 
    \] 
    where $R_H$ is a clarified $H$-Tambara functor, and at most one representative from each conjugacy class of subgroups of $G$ appears in the indexing set for the product.
\end{theorem}

\begin{proof}
	\cref{prop:product-splitting-of-G-rings} implies that we have a $G$-ring isomorphism \[ R(G/e) \cong \prod_H \mathrm{Coind}_H^G S_H \] where $S_H$ is a clarified $H$-ring. \cref{prop:bottom-level-reflects-products} implies \[ R \cong \prod_H T_H \] where $T_H$ is a $G$-Tambara functor such that $T_H(G/e) \cong \mathrm{Coind}_H^G S_H$. Finally, \cref{prop:coinduction-is-detected-at-bottom} implies \[ T_H \cong \mathrm{Coind}_H^G R_H \] for some $H$-Tambara functor $R_H$. Since $R_H(H/e) \cong S_H$, we see that $R_H$ is clarified.
	
	Suppose $H$ and $H'$ are conjugate. Then $\mathrm{Coind}_H^G R_H \cong \mathrm{Coind}_{H'}^G \mathrm{Res}_{H'}^H R_H$ where the restriction is along a $G$-set isomorphism $G/H \cong G/H'$. Since $R_H$ is clarified if and only if $\mathrm{Res}_{H'}^H R_H$ is, and coinduction commutes with products, we may rearrange and combine terms so that at most one representative from each conjugacy class of subgroups of $G$ appears in our product decomposition.
\end{proof}

This isomorphism is not canonical because it ultimately depends on some choices of idempotents in $R(G/e)$. In particular, it depends on the non-uniqueness of the decomposition of \cref{prop:product-splitting-of-G-rings}. Morally, \cref{prop:product-thm-diagonalizes} says that this is the only non-uniqueness, at least up to the difference between a constant presheaf and a locally constant sheaf.

We close this section with the observation that \cref{thm:product-decomposition} applies to most Tambara functors that arise in nature.

\begin{proposition}\label{prop:main-thm-applies-to-fp-over-noeth-base}
	Let $k$ be a Tambara functor with $k(G/e)$ a Noetherian ring. Then every finitely presented $k$-algebra is a product of coinductions of clarified $k$-algebras.
\end{proposition}

The Noetherian assumption is satisfied for example by the Burnside Tambara functor, any field-like Tambara functor, and any constant Tambara functor associated to a Noetherian ring. Essentially every Tambara functor that arises in nature satisfies this condition.

\begin{proof}
	Recall from \cite{SSW24} that a finitely presented $k$-algebra is precisely a coequalizer of two maps $k[X] \rightrightarrows k[Y]$ of free polynomial $k$-algebras (cf. \cref{def:free-poly-algs}). Observe from \cite{Bru05} and the fact that the $G/e$ level of the box product is given by the tensor product that we have an identification $k[X](G/e) \cong k(G/e)[X]$. Since $R \mapsto R(G/e)$ preserves colimits, the bottom level of this coequalizer is a finitely presented $k(G/e)$-algebra. By Hilbert's basis theorem any finitely presented algebra over a Noetherian ring is Noetherian, so that \cref{thm:product-decomposition} applies.
\end{proof}

There has not been much study towards the Noetherian condition for Tambara functors, by which we mean the ascending chain condition for Tambara ideals (although see \cite{CW25} for some related results). It seems possible that a Tambara functor is Noetherian if and only if each level is Noetherian. At the very least, one would hope that if $k$ is Noetherian, then so is $k(G/e)$.
\section{Clarification}

By \cref{thm:product-decomposition} every (sufficiently small) Tambara functor is a product of coinductions from proper subgroups, along with a clarified factor. One therefore expects projection onto the clarified factor to describe an interesting functor. Indeed this is the case; we will see in this section that this projection describes the subcategory of clarified Tambara functors as a reflective localization of the category of all Tambara functors.

We start with an analysis of morphisms out of Tambara functors, extracting some consequences of \cref{thm:product-decomposition}. Then we define the various clarification functors and show that they describe reflective localizations of categories of Tambara functors. Finally, we analyze the behavior of colimits in the category of clarified Tambara functors.

\subsection{Morphisms between coinduction and clarification}

\cref{thm:product-decomposition} gives a description of an (almost) arbitrary $G$-Tambara functor which is very amenable to receiving maps. In particular, mapping into a product is well-understood, and since coinduction is a right adjoint, mapping into coinductions is well-understood, especially since the left adjoint, restriction, admits a particularly simple description. On the other hand, surprisingly, we are able to establish some control on maps out of $G$-Tambara functors. As in the proof of \cref{thm:product-decomposition}, our analysis begins with an analysis of maps of $G$-rings.

\begin{lemma}\label{lem:idempotent-orbits-analysis}
	Let $f : R \rightarrow S$ be a morphism of $G$-rings and let $d$ be a type $H$ idempotent. Then $f(d)$ is either zero or a type $H$ idempotent.
\end{lemma}

\begin{proof}
	If $f(d)$ is nonzero, then the $G$-orbits of $f(d)$ are mutually orthogonal nonzero idempotents. Since $f$ is $G$-equivariant, the isotropy of $f(d)$ contains $H$, and since mutually orthogonal nonzero idempotents are distinct, there are at least $|G/H|$ elements in the $G$-orbit of $f(d)$. Thus the isotropy group is precisely $H$.
\end{proof}

\begin{lemma}\label{lem:no-Gring-maps-which-decrease-coinduction}
	Let $f : \mathrm{Coind}_H^G R_H \rightarrow \mathrm{Coind}_K^G R_K$ be a map of $G$-rings, where $R_K$ is a nonzero clarified $K$-ring. Then $K$ is subconjugate to $H$ in $G$.
\end{lemma}

\begin{proof}
	Note that $f$ is adjoint to a morphism \[ \mathrm{Res}_K^G \mathrm{Coind}_H^G R_H \rightarrow R_K \] and recall the isomorphism \[ \mathrm{Res}_K^G \mathrm{Coind}_H^G R_H \cong \prod_{g \in K \backslash G / H} \mathrm{Coind}_{K \cap {}^g H}^K \mathrm{Res}_{K \cap {}^g H}^{{}^g H} {}^g R \] of \cref{lem:arbitrary-res-of-coind-of-Gring}. Since every idempotent of $R_K$ is fixed, our morphism factors through projection onto the factors for which $K \cap {}^g H = K$. If no such factors exist, then our map factors through zero, contradicting the assumption that $R_K$ is nonzero. Therefore there exists $g \in G$ such that ${}^g H$ contains $K$, or, equivalently, $K$ is subconjugate to $H$ in $G$.
\end{proof}

Give the set $\mathrm{Sub}(G)$ of subgroups of $G$ a partial order under the relation $H \leq H'$ if a conjugate of $H$ is a subgroup of $H'$.

\begin{definition}
	Let $\Lambda$ be an upward closed subset of $\mathrm{Sub}(G)$. A $G$-ring is \emph{$\Lambda$-clarified} if whenever it contains an idempotent of type $L$, we have $L \in \Lambda$. A $G$-Tambara functor is $\Lambda$-clarified if its bottom level is a $\Lambda$-clarified $G$-ring.
\end{definition}

The upward closure property is essentially forced. One may show that if a $G$-ring $R$ contains an idempotent of type $H$, then it contains an idempotent of type $K$ whenever $H$ is subconjugate to $K$. Products of the rings $\mathrm{Coind}_H^G \Z$ show that every upward closed subset of $\mathrm{Sub}(G)$ is realized as the set of subgroups of $G$ for which there exists an idempotent of that type.

If $H$ is a subgroup of $G$, write $\Lambda_H$ upward closure of $H$ in $\mathrm{Sub}(G)$. A $G$-ring is clarified if and only if if is $\Lambda_G$-clarified. On the other extreme, every $G$-ring is $\Lambda_e$-clarified. Any $G$-ring of the form $\mathrm{Coind}_H^G R$ for $R$ a clarified $H$-ring is $\Lambda_H$-clarified by \cref{cor:best-clarification-bound-on-coind-of-clar} below. One important example of $\Lambda$-clarified Tambara functors are the free polynomial algebras over a $\Lambda$-clarified base (cf. \cref{def:free-poly-algs}).

\begin{lemma}\label{lem:free-poly-k-algs-are-clarified}
	Let $k$ be a $\Lambda$-clarified $G$-Tambara functor. All free polynomial $k$-algebras on finitely many generators are $\Lambda$-clarified.
\end{lemma}

\begin{proof}
	Every free polynomial Tambara functor on finitely many generators has level $G/e$ given by a finitely generated free polynomial $\mathbb{Z}$-algebra according to \cite[Theorem A]{Bru05}. The free polynomial $k$-algebras are obtained by base-changing the free polynomial Tambara functors along the unit map from the Burnside Tambara functor to $k$.
	
	Since the bottom level of the box product is the tensor product, we see that every free polynomial $k$-algebra on finitely many generators is has bottom level given by a free polynomial $k(G/e)$-algebra on finitely many polynomial generators. Any idempotent of such a ring must belong to $k(G/e)$ itself for degree reasons, hence must be fixed by the appropriate $H$.
\end{proof}

Note that \cref{def:free-poly-algs} as given in \cite{BH18} is given actually for incomplete Tambara functors. We would hope for \cref{lem:free-poly-k-algs-are-clarified} to be true for bi-incomplete Tambara functors, although an investigation of this would take us too far afield.

Next, we deduce from our results concerning the behavior of idempotents under equivariant ring maps some properties of maps between products of coinductions and $\Lambda$-clarified Tambara functors.

\begin{proposition}\label{prop:map-into-clarified-uniquely-factors-through-quotient}
	Any morphism of $G$-Green (resp. $G$-Tambara functors) $R \rightarrow S$ whose codomain is $\Lambda$-clarified factors uniquely through the quotient by the Green (res. Nakaoka) ideal generated by those idempotents in $R(G/e)$ of type $L$ for $L \in \Lambda$.
\end{proposition}

\begin{proof}
	This is an immediate consequence of \cref{lem:idempotent-orbits-analysis}.
\end{proof}

\begin{proposition}\label{prop:characterization-of-lambda-clarified}
	Let $\Lambda$ be an upward closed subset of $\mathrm{Sub}(G)$. If a Tambara functor $R$ is isomorphic to \[ \prod_{L \in \Lambda} \mathrm{Coind}_L^G R_L \] for $R_L$ a clarified $L$-Tambara functor, then $R$ is $\Lambda$-clarified. The converse holds if $R(G/e)$ is Noetherian.
\end{proposition}

\begin{proof}
	\cref{prop:map-into-clarified-uniquely-factors-through-quotient} applied to the identity implies the backwards direction (using \cref{thm:product-decomposition} and the Noetherian hypothesis). For the forwards direction, note that the class of $\Lambda$-clarified Tambara functors is closed under taking products. If $L \in \Lambda$, we must show $\mathrm{Coind}_L^G R_L$ is $\Lambda$-clarified whenever $R_L$ is clarified. In other words, the isotropy group of every idempotent of $\mathrm{Coind}_L^G R_L$ whose distinct orbits are mutually orthogonal is in $\Lambda$. Since $\Lambda$ is upward closed, it suffices to show that $L$ is subconjugate to the isotropy group of any such idempotent.
	
	Suppose $\mathrm{Coind}_L^G R_L$ contains an idempotent $d_H$ of type $H$. Additionally, let $\delta_L$ denote the type $L$ idempotent of $\mathrm{Coind}_L^G R_L$ corresponding to projection onto $R_L$. Note that $(g d_H) \delta_L$ is an idempotent of $R_L$, hence is fixed by $L$. Since $d_H \neq 0$, there exists $g \in G$ such that $g d_H \delta_L \neq 0$. Replacing $d_H$ with $g d_H$ at the cost of replacing $H$ with a conjugate, we may assume $d_H \delta_L \neq 0$. If $H$ does not contain $L$, then let $\gamma$ be an element of $L$ not in $H$. We compute $0 = (d_H \gamma d_H) \delta_L = (d_H \delta_L)(\gamma (d_H \delta_L))$ in $R_L$, so that $d_H \delta_L$ is a nonzero idempotent of $R_L$ which is not fixed by $\gamma$, contradicting the assumption that $R_L$ is clarified.
\end{proof}

\begin{corollary}\label{cor:best-clarification-bound-on-coind-of-clar}
	If $R_L$ is clarified, then $\mathrm{Coind}_L^G R_L$ is $\Lambda_L$-clarified. In other words, if $\mathrm{Coind}_L^G R_L$ contains an idempotent of type $H$, then $H$ contains a conjugate of $L$.
\end{corollary}

\begin{corollary}\label{cor:no-Tamb-maps-which-decrease-coinduction}
	Let $\Lambda \subset \mathrm{Sub}(G)$ be upward closed. Let \[ f : \prod_{H \subset G} \mathrm{Coind}_H^G R_H \rightarrow \prod_{H \in \Lambda} \mathrm{Coind}_H^G S_H \] a map of either $G$-rings or $G$-Tambara functors, where the $S_H$ are clarified $H$-rings or $H$-Tambara functors respectively. Then $f$ factors (uniquely) through the projection \[ \prod_{H \subset G} \mathrm{Coind}_H^G R_H \rightarrow \prod_{H \in \Lambda} \mathrm{Coind}_H^G R_H . \]
\end{corollary}

We may therefore adopt the heuristic that any morphism of Tambara functors whose $G/e$ levels are Noetherian is ``upper triangular". To be more precise, let $f : R \rightarrow S$ and write $R \cong \prod_H \mathrm{Coind}_H^G R_H$ and $S \cong \prod_H \mathrm{Coind}_H^G S_H$ with $R_H$ and $S_H$ clarified $H$-Tambara functors. Then since products and direct sums of Mackey functors are the same thing, the underlying Mackey functor morphism of $f$ may be viewed as a matrix with rows and columns labeled by the subgroups of $G$.

Each row of $f$ corresponds to a Tambara functor morphism $R \rightarrow \mathrm{Coind}_H^G S_H$. By \cref{prop:map-into-clarified-uniquely-factors-through-quotient} and the observation that $\mathrm{Coind}_H^G S_H$ is $\Lambda_H$-clarified, this map factors uniquely through the $\Lambda_H$-clarification of $R$. In other words, the entries for the row of $f$ in question which are indexed by $L \in \Lambda_H$ are zero. If $G = C_{p^n}$, then in a straightfoward manner $f$ may be made literally upper triangular, owing to the simple structure of the subgroup lattice of $C_{p^n}$. Otherwise $f$ is only ``close" to upper triangular.

In fact, we can push our analysis of morphisms a little further. Suppose $R \cong R_1 \times R_2$ is a product of two Tambara functors. Let $d_i$ denote the idempotent with the property $R_i(G/e) \cong d_i R(G/e)$. Consider an arbitrary morphism $f : R \rightarrow S$ of Tambara functors and observe that we have a $G$-ring splitting $S(G/e) \cong f(d_1)S(G/e) \times f(d_2)S(G/e)$. By \cref{prop:bottom-level-reflects-products}, we obtain a product splitting $S \cong S_1 \times S_2$ of Tambara functors under which $f$ becomes the product $f_1 \times f_2$, where $f_i : R_i \rightarrow S_i$ is obtained from $R \rightarrow S \rightarrow S_i$ by observing that this composition uniquely factors through the projection $R \rightarrow R_i$. If our Tambara functor morphism is actually an automorphism, then this analysis goes much further (see \cref{prop:product-thm-diagonalizes} below).

Now we apply the preceeding paragraph to the product decomposition of \cref{thm:product-decomposition}. Any morphism between Tambara functors whose $G/e$ levels are Noetherian therefore is determined by a product of morphisms out of coinductions of appropriately clarified Tambara functors. Consider $f : \mathrm{Coind}_H^G R \rightarrow S'$. Then \cref{cor:no-Tamb-maps-which-decrease-coinduction} along with the fact that coinduction commutes with products implies that $S'$ is in the image of $\mathrm{Coind}_H^G$, so we write $S' \cong \mathrm{Coind}_H^G S$ for some (not necessarily clarified) $H$-Tambara functor $S$.

Now that we have expressed the codomain of $f$ as a coinduction, we may apply the coinduction-restriction adjunction to deduce that $f$ is the same data as a morphism \[ \prod_{g \in H \backslash G / H} \mathrm{Coind}_{H \cap {}^g H}^H \mathrm{Res}_{H \cap {}^g H}^{{}^g H} {}^g R \cong \mathrm{Res}_H^G \mathrm{Coind}_H^G R \rightarrow S . \] If $H$ is normal, then the codomain is really just a product over $G/H$ of copies of $R$. In principal, this process may be reverse engineered to construct an arbitrary morphism of Tambara functors from those of the above form. 

Next, we give a list of examples, and an immediate, concrete application of our results so far. It is not tough to prove, even without the language of clarification and coinduction, although we take the simplicity and elegance of its statement as evidence that clarified Tambara functors are a natural object to consider.

\begin{proposition}\label{prop:examples}
	The following Tambara functors are clarified.
	\begin{enumerate}
		\item The Burnside Tambara functor of \cref{ex:Burnside-Tamb-functor}.
		\item The representation Tambara functor of \cref{ex:representation-Tamb-functor}.
		\item Any constant Tambara functor: those fixed-point Tambara functors associated to rings with trivial action.
		\item The fixed-point $G$-Tambara functor associated to any Galois field extension with Galois group $G$.
		\item Any free polynomial algebra over a clarified Tambara functor (cf. \cref{def:free-poly-algs}).
		\item The Tambara functor associated to the assignment $G/K \mapsto \mathrm{H}^{2*}(K;\mathbb{Z})$ of the group cohomology of $K$ with coefficients in $\mathbb{Z}$ (with the trivial action) \cite{Eve63} \cite{Tam93}.
		\item $\underline{\pi}_0(MU_G)$, for $MU_G$ the $G$-equivariant (homotopical) complex cobordism spectrum.
	\end{enumerate}
	Consequently if $G$ is nontrivial, there are no Tambara functor morphisms from a Nullstellensatzian Tambara functor, and more generally any coinduced Tambara functor, to one of the above form.
\end{proposition}

\begin{proof}
	We observe that each item is clarified in order. First, $G$ acts trivially on the $G/e$ level of the first three kinds Tambara functors on the list, so they are clarified. Next, The fixed-point Tambara functor associated to a Galois field extension has bottom level a field, hence the only idempotents are $1$ and $0$, which are fixed by $G$. That free polynomial algebras are clarified over a clarified base is the content of \cref{lem:free-poly-k-algs-are-clarified}. Next, we compute $\mathrm{H}^{2*}(e;\mathrm{Z}) \cong \mathrm{Ext}_{\Z}^{2*}(\Z,\Z)$ which is just $\Z$ in degree zero, and $\underline{\pi}_0(MU_G)(G/e) \cong \pi_0(MU) \cong \Z$. Any Tambara functor $k$ with $k(G/e) = \Z$ is clarified, so we deduce that the last two list items are clarified.
	
	The main theorem of \cite{SSW24} asserts that every Nullstellensatzian Tambara functor is the coinduction of an algebraically closed field. Therefore our assertion on morphisms out of these follows from \cref{prop:map-into-clarified-uniquely-factors-through-quotient}, the fact that the clarification of anything coinduced is zero, and the fact that any morphism of Tambara functors whose domain is zero is an isomorphism.
\end{proof}

Another natural and important source of coinduced Tambara functors are the free modules on a single generator (c.f. \cref{prop:free-k-modules-are-coinduced}). Specifically, if $k$ is a Tambara functor, then the free module over the Green functor underlying $k$ on a single generator in level $G/H$ arises as the $k$-module structure on $\mathrm{Coind}_H^G \mathrm{Res}_H^G k$ induced by the adjunction unit 
\[ 
	k \rightarrow \mathrm{Coind}_H^G \mathrm{Res}_H^G k .
\]
In this situation, we really only needed the Green functor structure on the target, although this adjunction unit is a morphism of Tambara functors.

This subsection concludes with an analysis of the non-uniqueness of the product decomposition of \cref{thm:product-decomposition}.

\begin{proposition}\label{prop:product-thm-diagonalizes}
	When the product decomposition of \cref{thm:product-decomposition} holds, it diagonalizes any Tambara functor automorphism.
\end{proposition}

\begin{proof}
	Let $\phi$ be an automorphism of $R := \prod_{H \subset G} \mathrm{Coind}_H^G R_H$ with $R_H$ clarified and write $\phi_H$ for the factor $R \rightarrow \mathrm{Coind}_H^G R_H$ of $\phi$. \cref{cor:best-clarification-bound-on-coind-of-clar} and \cref{prop:map-into-clarified-uniquely-factors-through-quotient} imply that $\phi_H$ factors through the projection \[ R \rightarrow \prod_{L \in \Lambda_H} \mathrm{Coind}_L^G R_L . \] We will show $\phi_H$ actually factors through projection onto $\mathrm{Coind}_H^G R_H$. 
	
	Let $d_L$ denote the idempotent of $R(G/e)$ corresponding to projection onto the factor $( \mathrm{Coind}_L^G R_L )(G/e)$. Then $\phi_H$ induces $G$-ring maps \[ \mathrm{Coind}_L^G (R_L(L/e)) \rightarrow \phi_H(d_L) \mathrm{Coind}_H^G R_H(H/e) . \] Suppose $\phi_H(d_L) \neq 0$. Then the codomain contains a type $H$ idempotent because coinduction commutes with limits. This type $H$ idempotent lifts to a type $H$ idempotent of $R(G/e)$. Since $\phi$ is injective, this type $H$ idempotent determines a type $H$ idempotent in $\mathrm{Coind}_L^G R_L(L/e)$, which implies that $L$ is subconjugate to $H$. We already know $H$ is subconjugate to $L$ from the first paragraph, and since at most one representative from each conjugacy class of subgroups of $G$ appears in the product decomposition of \cref{thm:product-decomposition}, we observe $H = L$.
\end{proof}

\subsection{Clarification as a localization}

We will see that the projections appearing in \cref{prop:map-into-clarified-uniquely-factors-through-quotient} determine reflective localizations on the category of $G$-Tambara functors. While these localizations make sense in any category of bi-incomplete $G$-Tambara functors, \cref{prop:characterization-of-lambda-clarified} implies that these localizations are most explicitly understandable in the Tambara case.

\begin{proposition}
	The subcategory $G {-} \Tamb_{\Lambda {-} \mathrm{clar}}$ of $\Lambda$-clarified $G$-Tambara functors is a reflective subcategory of the category of $G$-Tambara functors, i.e. the inclusion has a left adjoint.
\end{proposition}

\begin{proof}
	The left adjoint sends a $G$-Tambara functor $R$ to the quotient of $R$ by the ideal generated by those idempotents in $R(G/e)$ of type $K$ for each $K \notin \Lambda$. By \cref{prop:map-into-clarified-uniquely-factors-through-quotient}, this is indeed a left adjoint.
\end{proof}

This proposition also applies in the relative setting. In particular, a $k$-algebra is $\Lambda$-clarified if it is clarified as a Tambara functor. The inclusion of $\Lambda$-clarified $k$-algebras is a reflective subcategory.

\begin{definition}
	Let $R$ be a $k$-algebra. Define its \emph{$Lambda$-clarification} as the image under the left adjoint to the inclusion of $\Lambda$-clarified $k$-algebras in all $k$-algebras. In particular, the $\{ e \}$-clarification of $R$ is $R$ and the $\Lambda_G$-clarification of $R$ is clarified. When no $\Lambda$ is specified, we implicitly mean $\Lambda_G$-clarification.
\end{definition}

If $\Lambda' \supset \Lambda$, then any $\Lambda$-clarified Tambara functor $R$ is $\Lambda'$-clarified. Therefore as $\Lambda$ varies through upward closed subsets of $\mathrm{Sub}(G)$, we obtain a nested sequence of subcategories consisting of the $\Lambda$-clarified $k$-algebras. The $\Lambda$-clarification functors thus determine a cofiltration (indexed by upward closed subsets of $\mathrm{Sub}(G)$) of the category of $G$-Tambara functors. In the presence of all norms and the relatively mild Noetherian assumption of \cref{thm:product-decomposition}, we may understand the associated graded pieces of this cofiltration: they are precisely the images of coinduction functors.

\begin{proposition}\label{prop:SES-for-Lambda-clarification}
	If $k(G/e)$ is Noetherian, the $\Lambda$-clarification of a finitely presented $k$-algebra $R$ is the zero $k$-algebra if and only if $R$ is a product of coinductions \[ R \cong \prod_{H \notin \Lambda} \mathrm{Coind}_H^G R_H . \]
\end{proposition}

Note that we are not assuming $R_H$ is clarified.

\begin{proof}
	The backwards direction follows from the visible existence of idempotents with certain isotropy in $\prod_{H \notin \Lambda} \mathrm{Coind}_H^G R_H$. The forwards direction follows straightfowardly from \cref{prop:characterization-of-lambda-clarified}.
\end{proof}

\cref{prop:SES-for-Lambda-clarification} is a precise formulation of the slogan ``$\Lambda$-clarification is given by setting the image of coinduction equal to zero". If $G = C_{p^n}$ the only clarification functors available to us are the $\Lambda_{C_{p^m}}$-clarifications, and a $C_{p^n}$-Tambara functor maps to zero under such a clarification if and only if it is in the image of coinduction from $C_{p^m}$-Tambara functors (assuming the $C_{p^n}/e$ level is Noetherian). 

In principle the behavior of clarification on the category of $k$-algebras should allow one to extract information about $k$ itself. We give a simple application of this philosophy below. To declutter notation, we use $\mathrm{Coind}_a^b$ to mean coinduction along the subgroup inclusion $C_{p^a} \subset C_{p^b}$.

\begin{proposition}
	Let $G = C_{p^n}$ and assume $k(C_{p^n}/e)$ is Noetherian. The $\Lambda_{C_{p^m}}$-clarification functor on the category of finitely presented $k$-algebras is identically zero if and only if $k \cong \mathrm{Coind}_{m-1}^n \ell$ for some $C_{p^{m-1}}$-Tambara functor $\ell$.
\end{proposition}

\begin{proof}
	\cref{prop:SES-for-Lambda-clarification} implies that $k \cong \mathrm{Coind}_{m-1}^n \ell$ if and only if the $\Lambda_{C_{p^m}}$-clarification of $k$ is zero. Now the identity map $k \rightarrow k$ is the initial object, and since $\Lambda_{C_{p^m}}$-clarification preserves initial objects, the $\Lambda_{C_{p^m}}$-clarification of $k$ is the initial $\Lambda_{C_{p^m}}$-clarified $k$-algebra. The only morphism out of the zero Tambara functor is zero, so $\Lambda_{C_{p^m}}$-clarification is identically zero if and only if the $\Lambda_{C_{p^m}}$-clarification of $k$ is zero.
\end{proof}

We give an application to \'{e}tale theory. We define a morphism of $G$-Tambara functors $k \rightarrow R$ to \'{e}tale if $R$ is a finitely presented $k$-algebra and $k \rightarrow R$ is formally \'{e}tale in the sense of \cite{Hil17} (i.e. $R$ is flat as a $k$-module, and the genuine K\"{a}hler differentials $\Omega^{1,G}_{R/k}$ vanish).

\begin{theorem}\label{thm:clarification-of-etale-is-etale}
	Let $k$ be a Tambara functor with $k(G/e)$ Noetherian and $R$ a $k$-algebra, and $k \rightarrow R^{\mathrm{clar}}$ the clarification of $R$. If $k \rightarrow R$ is \'{e}tale, then so is $k \rightarrow R^{\mathrm{clar}}$.
\end{theorem}

\begin{proof}
	This follows from \cref{thm:product-decomposition} and the fact that a product of $k$-algebras is \'{e}tale if and only if each factor is. For the reader's convenience, we sketch the argument. Write $R \cong \widehat{R} \times R^{\mathrm{clar}}$. If $R$ is finitely presented and flat over $k$, then it is straightforward to show that the same holds for $R^{\mathrm{clar}}$. Finally, one may show that the (Tambara) genuine K\"{a}hler differentials of $R$ over $k$ contain the genuine K\"{a}hler differentials of $R^{\mathrm{clar}}$ as a summand (indeed, the functor sending a $k$-algebra $R$ to the genuine K\"{a}hler differentials $\Omega^{1,G}_{R/k}$ preserves products), hence the vanishing of the former implies vanishing of the latter.
\end{proof}

Our argument relies rather heavily on \cref{thm:product-decomposition}. Despite this, one might guess that clarification preserves \'{e}tale morphisms in general, or at least in the presence of all transfers. To the author's knowledge, the notion of being formally smooth has not been defined in the bi-incomplete case, so the following conjecture might not be very well-posed in full generality.

\begin{conjecture}\label{conj:clar-preserves-etale-of-bi-incomplete}
	Clarification of bi-incomplete Tambara functors sends \'{e}tale morphisms to \'{e}tale morphisms. Or, at least, clarification of Green functors sends Green \'{e}tale morphisms to Green \'{e}tale morphisms.
\end{conjecture}

Finally, we present a counterexample to the statement of \cref{thm:product-decomposition} for Green functors. This counterexample however does not provide a counterexample to \cref{conj:clar-preserves-etale-of-bi-incomplete}.

\begin{example}\label{ex:product-splitting-fails-for-Green-functors}
	Let $S$ be any Noetherian ring. Define a $C_p$-Green functor $R$ by the Lewis diagram
\[ \begin{tikzcd}
	S \times S \arrow[d, "\mathrm{Res}"] \\
	\mathrm{Coind}_e^{C_p} S \arrow[u, bend left = 45, "\mathrm{Tr}"]
\end{tikzcd} \] 
where the transfer is the sum over $C_p$-orbits followed by inclusion of the left factor, and the restriction is projection onto the left factor followed by the diagonal inclusion. This may be realized as the Green functor product of $\mathrm{Coind}_e^{C_p} S$ with the inflation of $S$ (the Green functor with top level $S$ and bottom level zero).
	
	Since $R(C_p/e)$ is a finite product of Noetherian rings, it is Noetherian. However, its clarification is trivial, and it is not coinduced, because all restrictions in a coinduction $\mathrm{Coind}_e^G T$ are injective. Thus \cref{thm:product-decomposition} cannot hold for Green functors. Interestingly, since zero is \'{e}tale, this is not a counterexample to \cref{conj:clar-preserves-etale-of-bi-incomplete}.
\end{example}

The author is not sure how to remove the hypothesis that $R(G/e)$ is Noetherian in \cref{thm:product-decomposition}. Essentially the issue is that one needs to choose a ``largest" idempotent of type $e$ and split off the corresponding coinduced factor, and then inductively choose a ``largest" idempotent of type $H$, inducting on the order of $H$, and split off corresponding coinduced factors.

One would like to appeal to the correspondence between idempotents in a ring and clopen subspaces of its Zariski spectrum to choose a ``maximal" type $H$ idempotent, by appealing to Zorn's lemma. The issue, ultimately, is that the arbitrary intersection or union of nested clopen subspaces might not be clopen. It's possible that this issue is somehow circumvented by passing to type $H$ idempotents. In any case, the assumption that $R(G/e)$ is Noetherian is, ultimately, quite mild, and morally does not impact any of our applications.

\subsection{Clarified colimits}

By standard properties of reflective subcategories, colimits in the category of $\Lambda$-clarified $k$-algebras are computed by first computing the colimit of the same diagram but in the category of $k$-algebras, and then applying the localization functor. Since the category of Tambara functors has all limits, we obtain the following.

\begin{proposition}
	The category of $\Lambda$-clarified $G$-Tambara functors into Tambara functors admits all limits and colimits. Furthermore, limits are computed levelwise.
\end{proposition}

On the other hand, we will see that colimits are not in general preserved by $\Lambda$-clarification. In particular, the box product of $k$-algebras is the coproduct, but it is possible for the box product of clarified $k$-algebras to fail to be clarified (cf. \cref{prop:box-product-of-clarified-is-not-clarified}). Nevertheless, the situation improves for \emph{filtered} colimits.

\begin{proposition}
	$\Lambda$-clarification preserves filtered colimits. In other words, filtered colimits in $G {-} \Tamb_{\Lambda {-} \mathrm{clar}}$ are computed levelwise.
\end{proposition}

\begin{proof}
	Recall that filtered colimits of Tambara functors are computed levelwise. Since colimits in the category of $\Lambda$-clarified objects are computed by applying $\Lambda$-clarification to the corresponding colimit in $G {-} \Tamb$, it suffices to show that any colimit in $G {-} \Tamb$ of a diagram whose objects are all $\Lambda$-clarified $G$-rings is $\Lambda$-clarified.
	
	Let $F : I \rightarrow G {-} \Ring$ be any filtered diagram and let $R$ be the colimit. Suppose $d \in R$ is an idempotent of type $H$ for some $H \subset G$. Then, passing to a large enough index $i \in I$, $d$ is in the image of some $\delta \in F(i)$ under canonical map $F(i) \rightarrow R$. We see that, modulo the kernel of $F(i) \rightarrow R$, $\delta$ satisfies the requisite relations to be a type $H$ idempotent. 
	
	Observe that being a type $H$ idempotent amounts to satisfying a finite list of relations: $hd - d = 0$ for all $h \in H$, $d \cdot gd = 0$ for all $gH \in G/H$, and $d^2 - d = 0$. Thus there exists $j \in I$ and $i \rightarrow j$ such that the induced map $F(i) \rightarrow F(j)$ carries each relation $h \delta - \delta$, $\delta \cdot g \delta$, and $\delta^2 - \delta$ to zero in $F(j)$ (because these elements map to zero in the colimit $R$). The image of $\delta$ in $F(j)$ is therefore a type $H$ idempotent.
	
	By contrapositive, we see that the colimit of a diagram consisting of $G$-rings which do not contain any type $H$ idempotents cannot contain any type $H$ idempotents. In other words, a filtered colimit of $\Lambda$-clarified $G$-rings is $\Lambda$-clarified.
\end{proof}

We will require the following in our study of Nullstellensatzian clarified Tambara functors.

\begin{corollary}\label{cor:Lambda-clarified-is-compactly-generated}
	Let $k$ be any Tambara functor. Then the category of $\Lambda$-clarified $k$-algebras is compactly generated.
\end{corollary}

\begin{proof}
	If $k$ is not $\Lambda$-clarified, then the category of $\Lambda$-clarified $k$-algebras is equivalent to the category of algebras over the $\Lambda$-clarification of $k$ (this follows from the description of the $\Lambda$-clarification of $k$ as the universal $\Lambda$-clarified $k$-algebra via \cref{prop:map-into-clarified-uniquely-factors-through-quotient}). Thus we may assume $k$ is $\Lambda$-clarified. It now suffices to prove that all free polynomial $k$-algebras are $\Lambda$-clarified since these form compact generators for the category of $k$-algebras, but this is \cref{lem:free-poly-k-algs-are-clarified}.
\end{proof}

On the other hand, coproducts are much more badly behaved in the clarified setting.

\begin{example}\label{ex:clarification-study-for-deg-two-ext-of-Q}
	Consider the $C_2$-Tambara functors $\ell = \mathrm{FP}(\Q(\sqrt{-2}))$ where we give $\Q(\sqrt{-2})$ the Galois $C_2$-action and $\ell^{triv} = \mathrm{FP}(\Q(\sqrt{-2}))$ where we give $\Q(\sqrt{-2})$ the trivial $C_2$-action. Both admit unit structures as algebras over $\mathrm{FP}(\Q)$. We aim to study the box products of $\ell$ and $\ell^{triv}$ with themselves and each other.
	
	Recall the isomorphism $\Q(\sqrt{-2}) \otimes_{\Q} \Q(\sqrt{-2}) \cong \Q(\sqrt{-2})^{\times |C_2|}$ given by $\sqrt{-2} \otimes 1 \mapsto (\sqrt{-2},\sqrt{-2})$ and $1 \otimes \sqrt{-2} \mapsto (\sqrt{-2},-\sqrt{-2})$ (this is an algebra homomorphism because the tensor product is the coproduct, and it is an isomorphism because it is a surjection between rational vector spaces of dimension four). Letting $C_2$ act on each of the two copies of $\Q(\sqrt{-2})$, we obtain a diagonal $C_2$ action on their tensor product over $\Q$, which becomes a $C_2$ action on the product. If this action transitively permutes the factors, then it is coinduced as a $C_2$-ring, hence has trivial clarification.
	
	Note that $\sqrt{-2} \otimes 1 + 1 \otimes \sqrt{-2} \mapsto (2 \sqrt{-2},0)$. If we give both factors in $\Q(\sqrt{-2}) \otimes_{\Q} \Q(\sqrt{-2}) \cong \Q(\sqrt{-2})^{\times |C_2|}$ the trivial action, this element is fixed, and if we give both factors the Galois action, then the generator for $C_2$ acts by $-1$. However, if we give one factor the trivial action and the other factor the Galois action, then we observe that the factors of $\Q(\sqrt{-2})^{\times |C_2|}$ are swapped (possibly with an automorphism of $\Q(\sqrt{-2})$ applied depending on the choice of which factor we gave a nontrivial action to).
	
	In conclusion, we see that $\ell \boxtimes_{\mathrm{FP}(\Q)} \ell \cong \ell^{\times |C_2|}$, $\ell^{triv} \boxtimes_{\mathrm{FP}(\Q)} \ell^{triv} \cong \left( \ell^{triv} \right)^{\times |C_2|}$, and $\ell \boxtimes_{\mathrm{FP}(\Q)} \ell^{triv} \cong \mathrm{Coind}_e^{C_2} \Q(\sqrt{-2})$. In particular, the coproduct of $\ell$ and $\ell^{triv}$ in the category of clarified $\mathrm{FP}(\Q)$-algebras is zero!
\end{example}

We may generalize a part of this example. Specifically, there are more instances for which the coproduct in clarified Tambara functors is zero.

\begin{example}
	Let $K \rightarrow L$ be a Galois extension of fields with Galois group $G$. Then we can form the fixed-point $G$-Tambara functors $k$, $\ell$, and $\ell^{triv}$ associated respectively to the trivial action on $K$, the Galois $G$ action on $L$, and the trivial $G$ action on $L$. The inclusion $K \rightarrow L$ determines maps $k \rightarrow \ell$ and $k \rightarrow \ell^{triv}$, so we can consider the $k$-algebra $\ell \boxtimes_k \ell^{triv}$. The $G/e$ level of this $k$-algebra is isomorphic to $L \otimes_K L$ as a ring.
	
	Since $L$ is Galois over $K$, it is a separable finite field extension, hence has a primitive element $\alpha \in L$. Then we compute 
\[ 
	L \otimes_K L \cong \prod_{g \in G} L[x]/(x-g \alpha) \cong \prod_{g \in G} L .
\] 
The $G$ action has the following description: it is the pullback along the diagonal $G \rightarrow G \times G$ of the latter group acting as the external product of the Galois $G$-action on the left copy of $L$ and the trivial $G$-action on the right copy of $L$. According to the description of the idempotents in $\prod_{g \in G} L$ in terms of elements of $L \otimes_K L$ given in \cite{Vij21}, we therefore see that the $G$-action freely and transitively permutes the factors. In particular, $\ell \boxtimes_k \ell^{triv}(G/e) \cong \mathrm{Coind}_e^G L$ by \cref{prop:coinduction-is-detected-at-bottom}. Interestingly, this is a field-like Tambara functor, even though, for example, $\ell^{triv} \boxtimes_k \ell^{triv}$ is not field-like.
	
	Since $L$ is a field, both $\ell$ and $\ell^{triv}$ are clarified. However, we see from the previous paragraph that the $\Lambda$-clarification of $\ell \boxtimes_k \ell^{triv}$ is zero for every upward closed subset $\Lambda \neq \{ e \}$ of $\mathrm{Sub}(G)$. On the other hand, $\ell^{triv} \otimes_k \ell^{triv}$ is clearly clarified. In general, it is unclear whether or not $\ell \boxtimes_k \ell$ is clarified, although in the case of $\Q \subset \Q(\sqrt{-2})$, we see from \cref{ex:clarification-study-for-deg-two-ext-of-Q} that it is.
\end{example}

From this example we deduce the following.

\begin{proposition}\label{prop:box-product-of-clarified-is-not-clarified}
	Let $G$ be a finite group such that there exists a Galois field extension with Galois group $G$. Then there exists a $G$-Tambara functor $k$ and nonzero clarified $k$-algebras $\ell_1$ and $\ell_2$ such that the clarification of $\ell_1 \boxtimes_k \ell_2$ is zero.
\end{proposition}

In particular, \cref{prop:box-product-of-clarified-is-not-clarified} implies that the hypotheses of the following Lemma are satisfiable. In fact, we will use \cref{lem:admitting-maps-from-two-things-implies-zero} in our study of Nullstellensatzian clarified Tambara functors below.

\begin{lemma}\label{lem:admitting-maps-from-two-things-implies-zero}
	Suppose $\ell_1 \sqcup \ell_2 = 0$ in the category of clarified $k$-algebras and let $R$ be an $\ell_2$-algebra. Then the clarification of $\ell_1 \boxtimes_k R$ is zero.
\end{lemma}

\begin{proof}
	Note that $\ell_1 \boxtimes_k R$ becomes an algebra over $\ell_1 \boxtimes_k \ell_2$ by base-change of $\ell_2 \rightarrow R$ from $k$ to $\ell_1$. But $\ell_1 \boxtimes_k \ell_2$ has trivial clarification, hence so does any Tambara functor receiving a map from it.
\end{proof}

\section{Applications}

In this section we investigate how clarification and \cref{thm:product-decomposition} interact with some facets of Tambara functors.

To start with, we note a few applications which appear in other articles. \cite{Wis25c} studies finite affine \'{e}tale group schemes; coinductions give ``new" examples of finite \'{e}tale extensions which are minimal in the appropriate sense. If $k$ is a Tambara functor with $k(G/e)$ an algebraically closed field and $\ell(G/e)$ is a finite \'{e}tale $k$-algebra, then $k(G/e) \rightarrow \ell(G/e)$ is finite \'{e}tale. \cref{thm:product-decomposition} applies, showing that $\ell$ is a finite product of coinductions of Tambara functors with $G/e$-level isomorphic to the algebraically closed field $k(G/e)$. 

On the other hand, \cite{Wis25b} observes that coinduction induces an order-preserving bijection on Green ideals and Tambara ideals which restricts to a bijection of prime ideals. Consequently the adjunction units $k \rightarrow \Coind_H^G \Res_H^G k$ lead to a stratification of the Nakaoka spectrum of $k$.

One may also show that \cref{thm:product-decomposition} is suitably compatible with the transversal Tambara functors of Sulyma \cite{Sul23}. Additionally, assuming an analogue of the main result of \cite{Yan25} for all finite groups, one expects an analogue of \cref{thm:product-decomposition} to hold for $G$-$\mathbb{E}_\infty$-rings $R$ under the assumption that $\pi_0^G(R)$ is a Noetherian ring.

Finally, we conclude with a non-application. Recall that \emph{cohomological} Mackey functors may be characterized as those Mackey functors which are modules over the fixed-point Tambara functor $\mathrm{FP}(\Z)$ corresponding to the trivial action. Cohomological Mackey functors are so-named due to the fact that group cohomology tends to provide examples of these. One might hope for some interaction between the notion of being clarified as a Green functor and being cohomological as a Mackey functor, although \cref{ex:clarified-and-cohomological} below suggests that there isn't anything nice to say.

\begin{example}\label{ex:clarified-and-cohomological}
	Consider the $C_3$-Galois extension $\F_{27}$ of $\F_3$. Note that $\mathrm{FP}(\F_{27})$ and $\mathrm{Coind}_e^{C_3} \F_3$ are both Tambara functor and Green functor $\mathrm{FP}(\F_3)$-algebras, and this algebra structure equips each with the structure of an $\mathrm{FP}(\F_3)$-module. Additionally, $\mathrm{FP}(\F_{27})$ and $\mathrm{Coind}_e^{C_3} \F_3$ are cohomological, as they are modules over $\mathrm{FP}(\Z)$ by pulling back the module structure through the morphism $\mathrm{FP}(\Z) \rightarrow \mathrm{FP}(\F_3)$.
	
	Observe in fact that $\mathrm{FP}(\F_{27})$ and $\mathrm{Coind}_e^{C_3} \F_3$ are isomorphic \emph{as modules}. Indeed, both are isomorphic to the $\mathrm{FP}(\F_3)$-module whose bottom level is the permutation representation $\F_3[C_3/e]$ and whose top level is $\F_3$, with restriction the inclusion of fixed-points and transfer forced by the double coset formula. However, $\mathrm{FP}(\F_{27})$ and $\mathrm{Coind}_e^{C_3} \F_3$ are not isomorphic as, for example one is clarified, whereas the other is not.
\end{example}

\subsection{\texorpdfstring{Morita invariance and $K$-theory}{Morita invariance and K-theory}}

Let $R$ be a $G$-Tambara functor. An $R$-module is defined to be a module over the underlying Green functor of $R$. We start by making precise an earlier claim that coinductions often arise ``in nature" as free modules, which is well-known to experts. A proof may be found in, for example, \cite{CW25}.

\begin{proposition}\label{prop:free-k-modules-are-coinduced}
	Let $k$ be a Green functor. The free $k$-module on a single generator in level $G/H$ is (isomorphic to) the $k$-module underlying the $k$-algebra given by the adjunction unit
	\[
		k \rightarrow \mathrm{Coind}_H^G \mathrm{Res}_H^G k .
	\]
\end{proposition}

If $R(G/e)$ is Noetherian, then by \cref{thm:product-decomposition} $R$ is a product of coinductions. Motivated by this, we study the interaction between modules, products, and coinductions.

\begin{proposition}\label{prop:product-induces-module-cat-iso}
	Let $R$ and $S$ be Green functors. Then the category of $R \times S$-modules is equivalent to the direct sum of the category of $R$-modules with the category of $S$-modules as symmetric monoidal abelian categories.
\end{proposition}

\begin{proof}
	Let $f$ and $g$ denote the projections of $R \times S$ onto $R$ and $S$ respectively. Then restriction along $f$ and $g$ respectively define functors from the categories of $R$ and $S$-modules to the category of $R \times S$-modules, hence define an additive functor \[ f^* \oplus g^* : R {-} \Mod \oplus S {-} \Mod \rightarrow R \times S {-} \Mod \] which we aim to show is a symmetric monoidal equivalence. Note that $f^* \oplus g^*$ is visibly faithful.
	
	Since limits are computed levelwise, we have $(R \times S)(G/H) \cong R(G/H) \times S(G/H)$. One checks straighftorwardly using the Dress pairing description of relative box products (cf. \cite{CW25}) that this induces an $R \times S$-module isomorphism $M \cong M' \oplus M''$ for any $R \times S$-module $M$, where $M'$ is in the image of restriction along $f$ and $M''$ is in the image of restriction along $g$. Therefore $f^* \oplus g^*$ is essentially surjective, and one sees that every morphism $f^* M \oplus g^* N \rightarrow f^* M' \oplus g^* N'$ is diagonal, i.e. $f^* \oplus g^*$ is full.
	
	Finally, we show that this equivalence is symmetric monoidal. This follows from the fact that the box product of $R \times S$-modules commutes with direct sums and the fact that the unit object is clearly preserved by our equivalence.
\end{proof}

Now we move on to coinduction. The statement of our main result here was suggested by Mike Hill who explained that it should be viewed as the discrete version of the main theorem of \cite{BDS15}. Special cases of this result appear in previous articles. Namely, the special case $H = e$ appears in \cite{SSW24} and the special case of $G$ abelian appears in \cite{CW25}.

\begin{proposition}\label{prop:coind-induces-module-cat-iso}
	Let $H \subset G$ and $R$ be any $H$-Green functor. Then coinduction induces a symmetric monoidal equivalence \[ R {-} \Mod \rightarrow \mathrm{Coind}_H^G R {-} \Mod . \]
\end{proposition}

\begin{proof}
	Let $M$ be a module over $\mathrm{Coind}_H^G R$. Recall the isomorphism \[ \mathrm{Coind}_H^G R(G/L) \cong \prod_{g \in H \backslash G / L} R(H/H \cap {}^g L) \] arising from \cref{lem:arbitrary-res-of-coind-of-HTamb} and let $d_{g,L}$ denote the idempotent corresponding to projection onto the factor $R(H/H \cap {}^g L)$. We therefore have a splitting \[ M(G/L) \cong \prod_{g \in H \backslash G / L} d_{g,L} M(G/L) . \] Letting $L$ run through the subgroups of $H$, define $N(H/L) := d_{e,L} M(G/L)$. We have a projection $\mathrm{Res}_H^G M \rightarrow N$ of Mackey functors which is evidently adjoint to $\phi : M \rightarrow \mathrm{Coind}_H^G N$. Note that $\mathrm{Coind}_H^G N$ is a $\mathrm{Coind}_H^G R$-module by virtue of \cref{lem:coind-is-lax-monoidal}, and that $\phi$ is a module morphism.
	
	First, observe that $\phi$ becomes an isomorphism after applying $\mathrm{Res}_H^G$. The argument is essentially identical to the first half of the argument of \cref{prop:coinduction-is-detected-at-bottom}, except that we replace the ring structure everywhere with the structure of a module over $\mathrm{Coind}_H^G R$; the point is that the conjugations in $\mathrm{Coind}_H^G R$ permute the idempotents which pick out the various factors of $M(G/L)$, hence the conjugations in $M$ determine the appropriate isomorphisms between the factors obtained after applying $\mathrm{Res}_H^G$. The argument of \cref{prop:coinduction-is-detected-at-bottom} also goes through, and shows that $\phi$ is itself an isomorphism; the main point is that injectivity is forced by Frobenius reciprocity in exact the same way, using the module structure rather than the ring structure. Finally, one checks that $\phi$ is an isomorphism of $\mathrm{Coind}_H^G R$-modules. Therefore $\mathrm{Coind}_H^G$ is essentially surjective.
	
	A morphism $\mathrm{Coind}_H^G M \rightarrow \mathrm{Coind}_H^G N$ of $\mathrm{Coind}_H^G R$-modules corresponds by adjunction to a morphism $\mathrm{Res}_H^G \mathrm{Coind}_H^G M \rightarrow N$ of $\mathrm{Res}_H^G \mathrm{Coind}_H^G R$-modules, where $N$ is viewed as a $\mathrm{Res}_H^G \mathrm{Coind}_H^G R$-module via restriction along projection onto $R$ (the $R$-module structure on $N$ arises from the first paragraph). The product of coinduction-restriction adjunction units makes $\mathrm{Res}_H^G \mathrm{Coind}_H^G R$ into an $R$-algebra (via the formula of \cref{lem:arbitrary-res-of-coind-of-HTamb}). Finally, one checks that the data of a $\mathrm{Res}_H^G \mathrm{Coind}_H^G R$-module morphism $\mathrm{Res}_H^G \mathrm{Coind}_H^G M \rightarrow N$ is equivalent to the data of an $R$-module morphism $M \rightarrow N$, where the domain $M$ is the identity double coset factor of $\mathrm{Res}_H^G \mathrm{Coind}_H^G M$. Thus $\mathrm{Coind}_H^G$ is fully faithful.
	
	Finally, we must show that coinduction is symmetric monoidal. It clearly preserves the unit. By our equivalence of categories, $\mathrm{Coind}_H^G M \boxtimes_{\mathrm{Coind}_H^G R} \mathrm{Coind}_H^G N$ is in the image of coinduction. In fact, it is the coinduction of the factor corresponding to the projection determined by the identity double-coset of $\mathrm{Res}_H^G \mathrm{Coind}_H^G R$ (using the equivalence of categories of \cref{prop:product-induces-module-cat-iso} and the formula of \cref{lem:arbitrary-res-of-coind-of-HTamb}). By \cite[Lemma 6.6]{Cha24} the restriction to $H$ is strong symmetric monoidal, so we obtain
\begin{align*}
	\mathrm{Res}_H^G & \left( \mathrm{Coind}_H^G M \boxtimes_{\mathrm{Coind}_H^G R} \mathrm{Coind}_H^G N \right) \\
	& \cong \mathrm{Res}_H^G \mathrm{Coind}_H^G M \boxtimes_{\mathrm{Res}_H^G \mathrm{Coind}_H^G R} \mathrm{Res}_H^G \mathrm{Coind}_H^G N 
\end{align*}
and one checks that the projection of the right-hand factor onto the identity double coset factor yields $M \boxtimes_R N$.
\end{proof}

\begin{definition}
	Let $\mathbf{R} = \{ R_i \}$ and $\mathbf{S} = \{ S_j \}$ be two finite sets with $R_i$ a $G_i$-Green functor and $S_j$ a $G_j$-Green functor. We say $\mathbf{R}$ and $\mathbf{S}$ are Morita equivalent if there is an equivalence of abelian categories \[ \oplus_i R_i {-} \Mod \cong \oplus_j S_j {-} \Mod \]
\end{definition}

\begin{theorem}\label{thm:Morita-equiv}
	Let $R$ be a $G$-Tambara functor with $R(G/e)$ Noetherian. Then $R$ is Morita equivalent to a set of clarified $H$-Tambara functors where $H$ runs through the subgroups of $G$. In particular, the $K$-theory of $R$ is a direct sum of $K$-theories of clarified $H$-Tambara functors.
\end{theorem}

\begin{proof}
	Combine \cref{thm:product-decomposition} with \cref{prop:product-induces-module-cat-iso} and \cref{prop:coind-induces-module-cat-iso}.
\end{proof}

In the special case for which $R$ cannot be a product of two nonzero Tambara functors, we obtain even better results.

\begin{corollary}\label{cor:K-theory-of-field-like}
	Let $R$ be any field-like $G$-Tambara functor. Then the $K$-theory of $R$ is isomorphic to the $K$-theory of a clarified field-like $H$-Tambara, for some $H \subset G$.
\end{corollary}

\begin{proof}
	This follows immediately from \cref{cor:field-likes-are-coinduced-from-clarified}, \cref{prop:coind-induces-module-cat-iso}, and the fact that $K$-theory only depends on the category of modules (in particular, the exact category of finitely generated projective modules, which is preserved by Morita equivalence).
\end{proof}

In particular, note that Chan and the author in \cite{CW25} study the algebraic $K$-theory of clarified Tambara functors (as a subclass of the class of Green meadows). Since every field-like Tambara functor is coinduced from a clarified field-like Tambara functor, the results in \cite{CW25} actually apply to a much larger class of field-like Tambara functors.

\cref{prop:coind-induces-module-cat-iso} may be pushed further, to the category of algebras in Green or Tambara functors. Because Green functors are precisely commutative monoids over the box product, we immediately deduce an analogue of \cref{prop:coind-induces-module-cat-iso} for Green functors (namely, if $k$ is a Green functor, then the category of $k$-algebras is equivalent to the category of $\mathrm{Coind}_H^G k$-algebras, in Green functors). Before we can prove the analogous result for Tambara functors, we require a lemma.

\begin{proposition}\label{prop:coind-has-unique-Tamb-structure}
	There is a bijection between Tambara functor structures on the Green functor $S$ and Tambara functor structures on the Green functor $\mathrm{Coind}_H^G S$.
\end{proposition}

\begin{proof}
	Any Tambara functor structure on the Green functor $S$ determines a Tambara functor structure on the coinduction $\mathrm{Coind}_H^G S$ by definition. Conversely, any Tambara functor structure on $\mathrm{Coind}_H^G S$ determines a Tambara functor structure on $S$ as follows. First, observe $\mathrm{Res}_H^G \mathrm{Coind}_H^G S$ is a product of $H$-Tambara functors (this follows from the fact that restriction takes $G$-Tambara functors to $H$-Tambara functors, and the Green functor product splitting of $\mathrm{Res}_H^G \mathrm{Coind}_H^G S$ promotes to a splitting of Tambara functors by \cref{prop:bottom-level-reflects-products}). The identity double coset factor in the product splitting is isomorphic to $S$ as a Green functor. Therefore any Tambara structure on $\mathrm{Coind}_H^G S$ determines a Tambara structure on $S$.
	
	Now we check that these constructions are inverse to each other. If we start with a Tambara structure on $S$, we clearly get back the Tambara structure with which we started (say, by \cref{lem:arbitrary-res-of-coind-of-HTamb}). On the other hand, suppose we have a Tambara functor structure on $\mathrm{Coind}_H^G S$ which induces a Tambara functor structure on $S$. Then the Tambara functor morphism $\mathrm{Res}_H^G \mathrm{Coind}_H^G S \rightarrow S$ is adjoint to a Tambara functor morphism $\mathrm{Coind}_H^G S \rightarrow \mathrm{Coind}_H^G S$, where the domain has the Tambara functor structure we started with, and the codomain has the Tambara functor structure coinduced from the Tambara functor structure on $S$. In the category of Green functors this morphism is the adjoint of the adjoint of the identity, hence is the identity morphism. As it is a morphism of Tambara functors, we see that our two Tambara functor structures on $\mathrm{Coind}_H^G S$ coincide.
\end{proof}

The special case $H = e$ of the result below was proven in \cite{SSW24}.

\begin{corollary}\label{cor:coind-gives-equiv-of-Tamb-algebra-cats}
	Coinduction induces an equivalence of categories \[ H {-} \Tamb_{k/} \cong G {-} \Tamb_{\mathrm{Coind}_H^G k /} \]
\end{corollary}

\begin{proof}
	By \cref{prop:coind-induces-module-cat-iso} and the fact that Green functors are precisely the commutative monoids with respect to the box product, we observe that coinduction induces an equivalence between the respective algebra categories of Green functors. By \cref{prop:coind-has-unique-Tamb-structure}, $\mathrm{Coind}_H^G : H {-} \Tamb_{k/} \cong G {-} \Tamb_{\mathrm{Coind}_H^G k}$ is essentially surjective. Faithfulness follows from faithfulness in the Green functor setting and the fact that the forgetful functor from Tambara to Green functors is faithful.
	
	Finally, we must show fullness. Any map $\mathrm{Coind}_H^G R \rightarrow \mathrm{Coind}_H^G S$ of Tambara functor $\mathrm{Coind}_H^G k$-algebras is also a map of Green functors, hence by the equivalence of categories for Green functors, arises from a morphism $R \rightarrow S$ of Green functors. In fact, this morphism is obtained by applying $\mathrm{Res}_H^G$ and projecting onto the identity double coset factor. Since both of these steps are compatible with Tambara functor structure, we see that $R \rightarrow S$ is a Tambara functor morphism. Applying $\mathrm{Coind}_H^G$ to this yields the Tambara functor morphism we started with.
\end{proof}

\begin{remark}
	Let $G$ be abelian. Then $\mathrm{Coind}_e^G R$ has a $G$-action given by permuting the bottom level factors. This is a $G$-action by Tambara functor automorphisms. However, it is \emph{not} a $G$-action in the category of $\mathrm{Coind}_e^G \Z$-algebras in Green or Tambara functors. Thus the essential image of the forgetful functor from $\mathrm{Coind}_e^G \Z$-algebras to $G$-Tambara functors admits a $G$-action.
\end{remark}

We emphasize that all of the results of this subsection are properties of the category of modules over the Green functor underlying a Tambara functor which depend inherently on the Tambara functor structure.

\subsection{Nullstellensatzian clarified Tambara functors}\label{subsec:NSS}

In this subsection we analyze properties of Nullstellensatzian clarified Tambara functors. We begin with a short review.

\begin{definition}
	An object $x$ in a locally small cocomplete category $\mathcal{C}$ is \emph{compact} if the functor \[ \mathrm{Hom}_{\mathcal{C}}(x,-) : \mathcal{C} \rightarrow \Set \] preserves filtered colimits.
\end{definition}

For example, in rings, a $k$-algebra is compact if and only if it is finitely presented. In fact, this remains true in Tambara functors: \cite[Proposition 4.3]{SSW24} asserts that, if $k$ is a Tambara functor, then a $k$-algebra $R$ is compact (in the category of $k$-algebras) if and only if it is the coequalizer of a diagram $k[x_{G/H_i}]_{i \in I} \rightrightarrows k[x_{G/H_j}]_{j \in J}$ with $I$ and $J$ finite. In this case we again call $R$ finitely presented.

\begin{definition}[\cite{BSY22}]
	A presentable $\infty$-category $\mathcal{C}$ is \emph{Nullstellensatzian} if every compact nonterminal object of $\mathcal{C}$ admits a map to the initial object of $\mathcal{C}$. An object $k$ of $\mathcal{C}$ is Nullstellensatzian if the slice category $\mathcal{C}_{k /}$ is Nullstellensatzian.
\end{definition}

The idea is that Nullstellensatzian objects are abstractions of algebraically closed fields in a certain sense. Hilbert's Nullstellensatz implies that a ring is Nullstellensatzian if and only if it is an algebraically closed field. One might therefore hope that Nullstellensatzian objects in equivariant algebra are related to $G$-actions on algebraically closed fields. One might hope that this relationship would have implications for the inverse Galois problem, as there is usually a plentiful supply of Nullstellensatzian objects in any reasonable category of ``things that look like rings". The author acknowledges that the idea that there could be a relationship of this form is due to David Chan.

One of the main results of \cite{BSY22} is the identification of the Nullstellensatzian objects in certain chromatically interesting categories of commutative ring spectra as precisely certain Lubin-Tate theories. Of more immediate interest is the main result of \cite{SSW24}.

\begin{theorem}[\cite{SSW24}]
	A Tambara functor is Nullstellensatzian if and only if it is isomorphic to the coinduction of an algebraically closed field.
\end{theorem}

We may actually give a relatively short proof using the machinery developed in the present article, although the argument boils down to the same ideas as those in \cite{SSW24}. We also note that \cite{SSW24} develops enough machinery for the study of Nullstellensatzian objects that a second, completely different proof is given.

\begin{proof}
	Let $\F$ be an algebraically closed field. The category of Tambara functor $\mathrm{Coind}_e^G \F$-algebras is equivalent to the category of $\F$-algebras by \cref{cor:coind-gives-equiv-of-Tamb-algebra-cats}. The latter category is Nullstellensatzian, hence so is the former.
	
	On the other hand, suppose $R$ is a Nullstellensatzian Tambara functor. Consider the finitely presented $R$ algebra given as the quotient of the free $R$-algebra on a single generator $x_{G/e}$ in level $G/e$ by finite list of relations which force $\{ g \cdot x_{G/e} \}$ to be a complete set of orthogonal idempotents. This quotient is nonterminal as it admits a map to $\mathrm{Coind}_e^G R(G/e)$ classifying the idempotent corresponding to projection onto $R(G/e)$. Now this finitely presented $R$-algebra admits a section which sends $x_{G/e}$ to a type $G$ idempotent in $R$.
	
	Now let $x \in R(G/H)$ be nonzero. Then $R/(x)$ is a finitely presented $R$-algebra which does not admit a section. Therefore $R/(x)$ is terminal, i.e. $(x) = R$. We see therefore that $R$ is a field-like Tambara functor containing a type $G$ idempotent, so that by \cref{cor:field-likes-are-coinduced-from-clarified} $R \cong \mathrm{Coind}_e^G \F$ for some field $\F$. By \cref{cor:coind-gives-equiv-of-Tamb-algebra-cats} the category of $R$-algebras is equivalent to the category of $\F$-algebras, so that $\F$ is Nullstellensatzian. In particular, $\F$ is an algebraically closed field.
\end{proof}

This theorem is a little unsatisfying, however, since the essential reason that Nullstellensatzian Tambara functors are what they are is because coinductions do not admit enough interesting maps. Additionally, it seems to have nothing to do with finite group actions on fields. We therefore propose Nullstellensatzian \emph{clarified} Tambara functors as a possible correction to this issue. We begin by showing that they exist, and then study some of their properties.

\begin{proposition}
	Let $k$ be a $\Lambda$-clarified Tambara functor. The category of clarified $k$-algebras is weakly spectral in the sense of \cite[Definition A.1]{BSY22}
\end{proposition}

\begin{proof}
	The terminal object in the category of $\Lambda$-clarified $k$-algebras is the zero $k$-algebra. Clearly the terminal object is both strict and compact. Additionally, we observed in \cref{cor:Lambda-clarified-is-compactly-generated} that the category of $k$-algebras is compactly generated (hence presentable).
\end{proof}

By \cite[Proposition A.17]{BSY22} every nonzero $k$-algebra admits a morphism to a Nullstellensatzian $k$-algebra. Nullstellensatzian objects are, by definition, nonterminal. In particular, we see that Nullstellensatzian $\Lambda$-clarified $k$-algebras exist. Specializing further, we see that Nullstellensatzian clarified Tambara functors exist.

Let $\ell_1$ and $\ell_2$ be $k$-algebras as in \cref{lem:admitting-maps-from-two-things-implies-zero}, and let $R$ be a Nullstellensatzian clarified Tambara functor which receives a map from $\ell_2$ (note that such an $R$ will always exist by the results of \cite{BSY22}). Then $\ell_1 \boxtimes_k R$ has trivial clarification by \cref{lem:admitting-maps-from-two-things-implies-zero}. Additionally, there does not exist any $k$-algebra map $\ell_1 \rightarrow R$, as otherwise we could box product this map with the identity $R \rightarrow R$, obtaining a morphism $\ell_1 \boxtimes_k R \rightarrow R$ from a $k$-algebra with trivial clarification to a nonzero clarified $k$-algebra.

\begin{lemma}\label{lem:NSS-clarified-implies-MRC}
	Let $R$ be a Nullstellensatzian $\Lambda$-clarified Tambara functor. Then all restrictions of $R$ are injective.
\end{lemma}

\begin{proof}
	Assume $x \in R(G/H)$ restricts to zero in $R(G/K)$. Then $x$ restricts to zero in $R(G/e)$. We may form the coequalizer $R[x_{G/H}] \rightrightarrows R$ along the zero map and the map classifying the element $x \in R(G/H)$. This coequalizer, which we write $R/(x)$, is the quotient of $R$ by the ideal generated by $x$, and is a compact object in the category of $R$-algebras (not necessarily $\Lambda$-clarified). 
	
	Since $x$ restricts to zero in $R(G/e)$, one deduces that $\left( R/(x) \right)(G/e) \cong R(G/e)$, hence $R/(x)$ is $\Lambda$-clarified; it is thus a compact object in the category of $\Lambda$-clarified $R$-algebras. Since $R/(x)$ is a nonzero finitely presented $\Lambda$-clarified $R$-algebra and $R$ is Nullstellensatzian (in the category of $\Lambda$-clarified Tambara functors), the quotient map $R \rightarrow R/(x)$ admits a section. The composition $R \rightarrow R/(x) \rightarrow R$ is the identity, hence sends $x$ to $x$, but $x$ maps to zero in the middle term. Thus $x = 0$. 
\end{proof}

\begin{lemma}\label{lem:NSS-clarified-are-weird}
	Let $R$ be a Nullstellensatzian clarified Tambara functor. If $I$ is any proper ideal of $R$, then $R/I$ has trivial clarification.
\end{lemma}

\begin{proof}
	Let $I$ be a proper ideal of $R$ and let $x \in I(G/H)$. If the clarification of $R/(x)$ is nonzero, then we observe as in the proof of \cref{lem:NSS-clarified-implies-MRC} that $x = 0$. If $I$ is nonzero, we can choose $x \in I(G/H)$ nonzero, so that the clarification of $R/(x)$ is zero. Applying clarification to the quotient $R/(x) \rightarrow R/I$ shows that the clarification of $R/I$ receives a map from the zero Tambara functor, hence is zero.
\end{proof}

\cref{lem:NSS-clarified-are-weird} suggests that Nullstellensatzian clarified Tambara functors are close to being field-like. On the other hand, we will see in \cref{thm:NSS-clarified-field-like-are-boring} and \cref{cor:existence-of-non-field-like-NSS-clarifieds} below that field-like Nullstellensatzian clarified Tambara functors are rare.

We start with a lemma. Recall that the free polynomial Tambara functor $\mathcal{A}[x_{G/G}]$ on a generator in level $G/G$ contains an element which we abusively denote $x_{G/G}$ in level $G/G$ corresponding to the identity map under the natural isomorphism \[ \mathrm{Hom}(\mathcal{A}[x_{G/G}], \mathcal{A}[x_{G/G}]) \cong \mathcal{A}[x_{G/G}](G/G) . \]

\begin{lemma}\label{lem:bottom-level-of-free-poly-alg-on-top-level-gen}
	Let $R$ be a $G$-Tambara functor. The free $R$-algebra $R[x_{G/G}]$ on a single generator in level $G/G$ has $R[x_{G/G}](G/e) \cong R(G/e)[\mathrm{Res}_e^G(x_{G/G})]$.
\end{lemma}

\begin{proof}
	Since $R[x_{G/G}] \cong R \boxtimes \mathcal{A}[x_{G/G}]$, it suffices to prove the result for $\mathcal{A}[x_{G/G}]$.	Recall that $\mathcal{A}[x_{G/G}](G/e)$ is a free polynomial algebra on a single generator \cite[Theorem A]{Bru05}. Additionally, $\mathcal{A}[x_{G/G}](G/e)$ is given by the set of bispans $G/G \leftarrow X \rightarrow Y \rightarrow G/e$ modulo the equivalence relation generated by natural isomorphisms of diagrams which are the identity on the left and right endpoints. Addition is induced by disjoint union of the middle terms. We therefore see that the isomorphism $\mathcal{A}[x_{G/G}](G/e) \cong \Z[t]$ of \cite[Theorem A]{Bru05} identifies $at^b$ with the bispan \[ G/G \leftarrow \left( G/e \right)^{\sqcup ab} \rightarrow \left( G/e \right)^{\sqcup a} \rightarrow G/e \] where the middle map is a disjoint union of $a$ copies of the fold map $\left( G/e \right)^{\sqcup b} \rightarrow G/e$.
	
	Now observe that the ring $\mathcal{A}[x_{G/G}](G/e)$ is given as a set by the bispan morphisms from $G/G$ to $G/e$. The identity on $G/G$ in the category of bispans corresponds to $x_{G/G}$, by Yoneda's lemma. The restriction of $x_{G/G}$ is obtained by postcomposing the identity bispan with \[ G/G \leftarrow G/e \rightarrow G/e \rightarrow G/e . \] As we observed in the previous paragraph, this is precisely a polynomial generator.
\end{proof}

\begin{theorem}\label{thm:NSS-clarified-field-like-are-boring}
	Let $R$ be a Nullstellensatzian clarified $G$-Tambara functor which is field-like and such that the characteristic of the field $R(G/G)$ does not divide $|G|$. Then $R(G/e)$ is algebraically closed.
\end{theorem}

\begin{proof}
	Let $f$ be any polynomial with coefficients in $R(G/G)$. Since it is an $R$-algebra, $R[x_{G/G}](G/G)$ is an $R(G/G)$-module, hence $f(x_{G/G})$ defines an element of $R[x_{G/G}](G/G)$. Recall the identification $R[x_{G/G}](G/e) \cong R(G/e)[\mathrm{Res}_e^G(x_{G/G})]$ of \cref{lem:bottom-level-of-free-poly-alg-on-top-level-gen}.
	
	Since $R$ is clarified and field-like, $R(G/e)$ is a field, containing $R(G/G)$ as a subfield. Suppose that $f$ is irreducible over $R(G/e)$. Then \[ \left( R[x_{G/G}]/(f) \right)(G/e) \cong R(G/e)[\mathrm{Res}_e^G(x_{G/G})]/(f(\mathrm{Res}_e^G(x_{G/G}))) \] is a field, hence $R[x_{G/G}]/(f)$ is clarified. As it is visibly finitely presented, it admits an $R$-algebra map to $R$, which sends $\mathrm{Res}_e^G(x_{G/G})$ to a root of $f$ in $R(G/e)$. We have shown that every polynomial with coefficients in $R(G/G)$ admits a root in $R(G/e)$. Our characteristic assumption implies $R$ is a fixed-point Tambara functor, so that $R(G/e)$ is a Galois extension of $R(G/G)$ by Artin's Lemma \cite{Art42}.
	
	Let $L$ denote an algebraic closure of $R(G/G)$, and choose an embedding of $R(G/e)$ in $L$. Let $x \in L$ be arbitrary, and let $f$ be the minimal polynomial of $x$ over $R(G/G)$. Since $f$ has a root in $R(G/e)$ and $R(G/e)$ is Galois over $R(G/G)$, we see $x \in R(G/e)$. Thus $R(G/e) \rightarrow L$ is an isomorphism, so that $R(G/e)$ is algebraically closed.
\end{proof}

\begin{corollary}\label{cor:existence-of-non-field-like-NSS-clarifieds}
	Let $G$ be a nontrivial finite group and assume there exists a finite field extension $L$ over $K$ with Galois group $G$. If $\mathrm{char}(L) \neq 0$ or $G \neq C_2$ then there exists a Nullstellensatzian clarified $G$-Tambara functor which is not field-like.
\end{corollary}

By \cref{lem:NSS-clarified-implies-MRC}, $R$ is field-like if and only if $R(G/e)$ is a field. Thus \cref{cor:existence-of-non-field-like-NSS-clarifieds} is equivalent to the slightly stronger statement that there exists Nullstellensatzian clarified Tambara functors whose bottom level is not a field.

\begin{proof}
	Let $G$ act faithfully on a field $L$ and form the field-like $G$-Tambara functor $\mathrm{FP}(L)$. We have already observed that there exists a Tambara functor morphism $\mathrm{FP}(L) \rightarrow R$ where $R$ is Nullstellensatzian clarified. Since $\mathrm{FP}(L)$ is field-like, this morphism is injective, hence $G$ acts faithfully on $R(G/e)$. If $R$ were field-like, then \cref{thm:NSS-clarified-field-like-are-boring} implies that $R(G/e)$ is an algebraically closed field with faithful $G$-action. The Artin-Schreier theorem \cite{AS27a} \cite{AS27b} implies that the only way for a nontrivial finite group $G$ to act faithfully on an algebraically closed field is for the characteristic to be zero and $G = C_2$.
\end{proof}

\begin{problem}\label{problem:what-do-they-look-like?}
	Determine the structure of Nullstellensatzian clarified $G$-Tambara functors. In particular, in the situation of \cref{cor:existence-of-non-field-like-NSS-clarifieds}, what is the ring underlying the bottom level of the corresponding Nullstellensatzian clarified Tambara functor?
\end{problem}

The most optimistic guess is that the converse of \cref{thm:NSS-clarified-field-like-are-boring} holds. For example, one might hope that $\mathrm{FP}(\C)$ with $G = C_2$ acting by complex conjugation is a Nullstellensatzian clarified Tambara functor.

\begin{conjecture}\label{conj:FP-C-is-NSS-clarified}
	The following is a complete list of field-like Nullstellensatzian clarified Tambara functors.
	\begin{enumerate}
		\item $\mathrm{FP}(\F)$ for $\F$ an algebraically closed field with the trivial $G$-action.
		\item $\mathrm{FP}(\F)$ for $\F$ an algebraically closed field with $G$ acting by the restriction along a surjection $G \rightarrow C_2$, and $C_2$ acting faithfully.
	\end{enumerate}
\end{conjecture}

\bibliographystyle{alpha}
\bibliography{ref}

\end{document}